\documentclass[a4paper,oneside]{article}

\usepackage{amsmath}%
\usepackage{amssymb}%
\usepackage[latin1]{inputenc}
\usepackage{graphicx}
\usepackage[section]{placeins}

\usepackage{geometry}
\newtheorem{theorem}{Theorem}[section]

\newtheorem{definition}[theorem]{Definition}

\newtheorem{lemma}[theorem]{Lemma}

\newtheorem{proposition}[theorem]{Proposition}
\newtheorem{remark}[theorem]{Remark}

\newenvironment{proof}[1][Proof]{\textbf{#1.} }{\ \rule{0.5em}{0.5em}}

\newcommand{\refeqn}[1]{(\ref{#1})}

\newcommand{\matr}[0]{\operatorname{Mat}}
\newcommand{\spann}[0]{\operatorname{span}}

\begin{document}

\title{A symmetry-adapted numerical scheme for SDEs}

\author{Francesco C. De Vecchi\thanks{Institute for Applied Mathematics and HCM, Rheinische Friedrich-Wilhelms-Universit\"at Bonn, Endenicher Allee 60, 53115 Bonn, Germany, \emph{email: francesco.devecchi@uni-bonn.de}}, Andrea  Romano\thanks{Dip. di Matematica, Universit\`a degli Studi di Milano, via Saldini 50, Milano, \emph{email: andrea.romano4@studenti.unimi.it}}
 and Stefania Ugolini\thanks{Dip. di Matematica, Universit\`a degli Studi di Milano, via Saldini 50, Milano, \emph{email: stefania.ugolini@unimi.it}}}
\date{}

\maketitle
\begin{abstract}
We propose a geometric numerical analysis of SDEs admitting Lie symmetries which allows us to individuate  a
symmetry adapted coordinates system where the given SDE has notable invariant properties. An approximation
scheme preserving the symmetry properties of the equation is introduced. Our algorithmic  procedure is
applied to the family of general linear SDEs 
for which two theoretical estimates of the numerical forward error are established.
\end{abstract}

\section{Introduction}

The exploitation of special geometric structures in numerical integration of both ordinary and partial differential equations
(ODEs and PDEs) is nowadays a mature subject of the numerical analysis often called \emph{geometric numerical integration}
(see e.g. \cite{Hairer2006,Iserles2009,Leimkuhler2004,Special_iusse}). The importance of this research topic is due to the
fact that many differential equations in mathematical applications have some particular geometrical features such as for
example a conservation law, a variational origin, an Hamiltonian or symplectic structure, a symmetry structure etc.. The
development of geometrically adapted numerical algorithms permits to obtain suitable integration methods which both preserve
the qualitative properties of the integrated equations and have a more efficient numerical behaviour with respect to the
corresponding standard discretization schemes. \\
In comparison the study of geometric numerical integration of stochastic differential equations (SDEs) is not so well developed.
In the current literature the principal aims consist in producing numerical stochastic integrators which are able to preserve the
symplectic structure (see e.g. \cite{Anton2015,Milstein2002,Tao2010}), some conserved quantities (see e.g. \cite{Cohen2016,Hong2011,Malham2008})
 or the variational structure (see e.g. \cite{Bou-Rabee2009,Bou-Rabee2010,Holm2016,Wang2009}) of the considered SDEs. For the study of the algebraic structure of stochastic expansions in order to achive an optimal efficient stochastic integrators at all
 orders see  \cite{Ebrahimi-Fard2012}.
\\
Although the exploitation of Lie symmetries of ODEs and PDEs (see e.g. \cite{Olver}) to obtain better numerical integrators
is an active research topic (see e.g. \cite{Winternitz2016,Dorodnitsyn2011,Levi2006,Levi2011} and references therein), the
application of the same techniques in the stochastic setting to the best of our knowledge  is not yet pursued, probably because
the concept of symmetry of a SDE has been quite recently developed (see e.g. \cite{DeVecchiT,DMU,DMU2,DeVecchi1,Gaeta1,Lazaro_Ortega,Lescot,Grigoriev}).\\
In this paper we introduce two different  numerical methods taking advantage of the presence of Lie symmetries in a given  SDE in order to provide a more efficient  numerical integration of it.\\
We first propose the definition of an invariant numerical integrator for a symmetric SDE as a natural generalization of the
corresponding concept for an ODE. When one tries to construct general invariant numerical methods in the stochastic framework,
in fact, a non trivial problem arises. Since both the SDE solution as well as the Brownian motion  driving it are continuous
but not differentiable processes, the  finite differences discretization could not converge to the SDE solution.  We give some necessary and sufficient conditions in order that the two standard numerical methods for SDEs (the Euler and the Milstein schemes) are also invariant numerical methods. By using this result, in particular, we are able to identify a class of privileged coordinates systems where it is convenient to make the discretization procedure. \\
Our second numerical method, based on a well-defined change of the coordinates system,  is inspired by the standard techniques
of reduction and reconstruction of a SDE with a solvable Lie algebra of symmetries (see \cite{DMU2,Kozlov2010}). Indeed a SDE
with a solvable Lie algebra of symmetries can be reduced to a triangular system and, when the number of symmetries is sufficiently
high, the latter can be explicitly integrated. In the stochastic setting the explicit integration concept is of course a quite
different notion with respect to the deterministic one. Indeed the evaluation of an Ito integral, a necessary step in the reconstruction
of a reduced SDE, can  only be numerically implemented.\\
We apply our two proposed numerical techniques to the general  linear SDEs, being the first non-trivial class of symmetric
equations. In this case the two algorithmic methods can be harmonized in such a way to produce the same simple family of best coordinates
systems for the discretization procedure. Interestingly, the identified coordinate changes are closely related to the explicit
solution formula of linear SDEs. Although the integration formula of linear SDEs is widely known, it is certainly original  the
recognition of the proposed numerical scheme  for  linear SDEs as a particular implementation of a general procedure for SDEs with Lie symmetries. We finally point out  that the simple case of affine drift and diffusion coefficients plays an important role  since any SDE with real analytic drift and diffusion coefficients can be seen locally like it.\\
Moreover we  theoretically investigate the numerical advantages of the new numerical scheme for linear SDEs. More precisely we obtain
two estimates for the forward numerical error which, in presence of an equilibrium distribution,  guarantee that the constructed method
is numerically stable for any size of the time step $h$. This means that for any $h>0$ the error does not grow exponentially with the
maximum-integration-time $T$, but it remains finite for $T \rightarrow + \infty$. This property is not shared by standard explicit or
implicit Euler and Milstein methods. The obtained estimates can be considered original results mainly because the coordinate changes
involved in the formulation of our numerical scheme are strongly not-Lipschitz, and so the standard convergence theorems  can not be
applied. Our theoretical results are also numerically illustrated.\\
It is interesting to note that the main part of the theory, in particular the definitions of strong symmetry of a SDE and of numerical scheme, can be easily extended to Stratonovich type SDE driven by general noises (\cite{DeVecchiT}), for example by exploiting rough paths theory. Unfortunately, since the proofs of Theorem 5.1 and Theorem 5.2 use in essential way the (forward and bachward) Ito formula, the long terms estimates obtained here cannot be straightforwardly generalized to the rough paths driven SDEs framework. At the same time we think that some ideas in the proof of Theorem 5.2 can be suitable exploited to obtain pathwise estimates of long term error in the rough paths setting.
The article is structured as follows: in Section \ref{section_preliminaries} we recall the notion of strong symmetry of a SDE and we
describe the two standard discretization schemes used in the rest of the paper. In Section \ref{section_symmetries_numerical} we propose
two numerical procedures adapted with respect to the Lie symmetries of a SDE.  We apply the  proposed integration methods  to general one and two-dimensional linear SDEs in Section \ref{section_linear} . In Section \ref{section_estimate} some theoretical estimates showing the stability
and efficiency of our adapted-to-symmetries numerical schemes in linear-SDEs are proved. In the last section we expose some numerical
experiments confirming the theoretical estimates obtained in the previous section.

\section{Preliminaries}\label{section_preliminaries}

\subsection{Strong symmetries of SDEs}

In the following $M$ will be an open subset of $\mathbb{R}^n$, and $x^1,...,x^n$ will be the standard coordinates system of $M$. If $F:M \rightarrow \mathbb{R}^k$ we denote by $\nabla(F)$ the Jacobian of $F$ i.e. the matrix-valued function
$$\nabla(F)=(\partial_{x^i}(F^j))|_{\stackrel{j=1,...,k}{i=1,...,n}}.$$
Furthermore we can identify the vector fields $Y \in TM$ with the functions $Y:M \rightarrow \mathbb{R}^n$, and if $\Phi_M \rightarrow M'$ is a diffeomorphism (with $M' \subset \mathbb{R}^n$) we introduce the pushforward
$$\Phi_*(Y)=(\nabla(\Phi) \cdot Y) \circ \Phi^{-1}.$$

\begin{definition}
Let $(\Omega,\mathcal{F},\mathcal{F}_t,\mathbb{P})$ be a filtered probability space. Let $\mu$ and $\sigma$ be two smooth functions defined on $M$ and valued in  n vectors and $n \times m$ matrices respectively. A  solution to a SDE($\mu,\sigma$) is a pair $(X,W)$ of adapted processes such that

i) W is a $\mathcal{F}_t$-Brownian motion in $\mathbb{R}^m$;

ii) For $i=1,2,...,n$
\begin{equation}\label{equation_SDE}
X^{i}_t=X^{i}_0+\int_0^t \mu^i(X_s)ds+ \int_0^t \sum_{\alpha=1}^m \sigma^i_{\alpha}(X_s)dW_s^{\alpha}+.
\end{equation}
\end{definition}
\begin{remark}
In particular all the integrals are meaningful if a.s.:
$$\int_0^t \sum_{i,\alpha} (\sigma^{i}_{\alpha})^2(X_s)ds   < +\infty, \quad \int_0^t \sum_{i} |\mu^i(X_s)|ds  < +\infty$$
\end{remark}
\begin{definition}
A solution $(X,W)$ to a SDE($\mu,\sigma$) on $(\Omega,\mathcal{F},\mathcal{F}_t,\mathbb{P})$ is said a strong solution if $X$ is adapted to the filtration $\mathcal{F}^W_t$ generated by the BM $W$ and completed with respect to $\mathbb{P}$.
\end{definition}
Of course a solution $(X,W)$ is called a weak solution when it is not strong.\\
In this paper we  fix a Brownian motion $W$, that is we consider only strong solutions of a  SDE($\mu,\sigma$) and, consequently, we denote them simply by $X$. For a symmetry analysis via weak solutions of SDEs see \cite{DMU}.\\
A solution $X$ to a SDE($\mu,\sigma$) is a diffusion process admitting as infinitesimal generator:
$$L=\sum_{\alpha=1}^m\sum_{i,j=1}^n \frac{1}{2}\sigma^i_{\alpha}\sigma^j_{\alpha}\partial_{x^ix^j}+\sum_{i=1}^n\mu^i\partial_{x^i}.$$
It is particularly useful for obtaining stochastic differentials the following  celebrated formula.
\begin{theorem}[Ito formula]
Let $X$ be a solution of  the SDE $(\mu,\sigma)$ and let $f:M \rightarrow
\mathbb{R}$ be a smooth function. Then  $F=f(X)$ has the following stochastic differential
$$dF_t=L(f)(X_t)dt+\nabla(f)(X_t) \cdot \sigma(X_t) \cdot dW_t.$$
\end{theorem}
We recall  important definitions of symmetries of a SDE.
\begin{definition}[strong finite symmetry]
We say that a diffeomorphism $\Phi$ is a (strong finite) symmetry of the SDE $(\mu,\sigma)$ if for any solution $X$ to the SDE $(\mu,\sigma)$ also $\Phi(X)$ is a solution to the SDE $(\mu,\sigma)$.
\end{definition}
By using Ito formula it is immediate to prove the following result.
\begin{theorem}
A diffeomorphism $\Phi$ is a symmetry of the SDE $(\mu,\sigma)$ if and only if
\begin{eqnarray*}
L(\Phi)&=&\mu \circ \Phi\\
\nabla(\Phi) \cdot \sigma&=&\sigma \circ \Phi.
\end{eqnarray*}
\end{theorem}

It is well-known that vector fields acting as infinitesimal generators of one parameter transformation groups are the most important tools in Lie group theory.
\begin{definition}[strong infinitesimal symmetry]
 A vector field $Y$ is said a (strong infinitesimal) symmetry of the SDE $(\mu,\sigma)$ if the group of the local diffeomorphism $\Phi_a$ generated by $Y$ is a symmetry of the SDE $(\mu,\sigma)$ for any $a \in \mathbb{R}$.
\end{definition}
The following determining equations for (any) infinitesimal symmetries are well-known (see, e.g., \cite{Gaeta1}). For their generalization
to a weak solution case see \cite{DMU}.
\begin{theorem}[Determining equations]
A vector field $Y$ is an infinitesimal symmetry of the SDE $(\mu,\sigma)$ if and only if
\begin{eqnarray}
Y(\mu)-L(Y)&=&0 \label{determining_equation1}\\
\left[Y,\sigma_{\alpha}\right]&=&0. \label{determining_equation2}
\end{eqnarray}
where $\sigma_{\alpha}$ is the $\alpha$-column of $\sigma$ ($\alpha=1,...,m$) and $[\cdot.\cdot]$ are the standard Lie brackets between vector fields.
\end{theorem}

\subsection{Numerical integration of SDEs}\label{subsection_integration}

For convenience of the reader, we recall the two main numerical methods for simulating a SDE and a theorem on the strong convergence of these methods (for a detailed description see e.g. \cite{Kloeden}).\\
Consider the SDE having coefficients $(\mu,\sigma)$, driven by the Brownian motion $W$, and let $\{t_n\}_n$  be a partition of $[0,T]$. The \emph{Euler scheme} for the equation $(\mu,\sigma)$ with respect to the given partition  is provided  by the following sequence of random variables $X_n \in M$
$$X^i_n=X^i_{n-1}+\mu^i(X_{n-1}) \Delta t_n +\sum_{\alpha=1}^m \sigma^i_{\alpha}(X_{n-1}) \Delta W^{\alpha}_n,$$
where $\Delta t_n=t_n-t_{n-1}$ and $\Delta W^{\alpha}_n= W^{\alpha}_{t_n}-W^{\alpha}_{t_{n-1}}$.
The \emph{Milstein scheme} for the same equation $(\mu,\sigma)$ is instead constituted by the sequence of random variables $\bar{X}_n \in M$ such that
\begin{eqnarray*}
\bar{X}^i_n&=&\bar{X}^i_{n-1}+\mu^i(\bar{X}_{n-1}) \Delta t_n +\sum_{\alpha=1}^m \sigma^i_{\alpha}(\bar{X}_{n-1}) \Delta W^{\alpha}_n+\\
&&+\frac{1}{2}\sum_{j=1}^n\sum_{\alpha,\beta=1}^m\sigma^j_{\alpha}(\bar{X}_{n-1}) \partial_{x^j}(\sigma^i_{\beta})(\bar{X}_{n-1}) \Delta \mathbb{W}^{\alpha,\beta}_n,
\end{eqnarray*}
where $\Delta \mathbb{W}^{\alpha,\beta}_n=\int_{t_{n-1}}^{t_n}{(W^{\beta}_s-W^{\beta}_{t_{n-1}})dW^{\alpha}_s}$. We recall that when $m=1$ we have that
$$\Delta \mathbb{W}^{1,1}_n=\frac{1}{2}((\Delta W_n)^2-\Delta t_n).$$

\begin{theorem}\label{theorem_convergence}
Let us denote by $X_t$ the exact solution of a SDE $(\mu,\sigma)$ and by $X_N$ and $\bar{X}_N$  the N-step approximations according with Euler and Milstein scheme respectively. Suppose that the coefficients $(\mu,\sigma)$ are $C^2$ with bounded derivatives and put $t_n=\frac{nT}{N}$ and $h=\frac{T}{N}$. Then there exists a constant $C(T,\mu,\sigma)$ such that
$$\epsilon_N=\left(\mathbb{E}[\|X_T-X_N\|^2]\right)^{1/2} \leq C(T,\mu,\sigma) h^{1/2}.$$
Furthermore when the coefficients $(\mu,\sigma)$ are $C^3$ with bounded derivatives then there exists a constant $\bar{C}(T,\mu,\sigma)$ such that
$$\bar{\epsilon}_N=\left(\mathbb{E}[\|X_T-\bar{X}_N\|^2]\right)^{1/2} \leq \bar{C}(T,\mu,\sigma) h.$$
\end{theorem}
\begin{proof}
See Theorem 10.2.2 and Theorem 10.3.5 in \cite{Kloeden}.
\end{proof}\\

Theorem \ref{theorem_convergence} states that $X_N$ and $\bar{X}_N$ strongly converge in $L^2(\Omega)$ to the exact solution $X_T$ of the
SDE $(\mu,\sigma)$, where the order of the  convergence with respect to the step size variation  $h=\frac{T}{N}$ is $\frac{1}{2}$ in the
Euler case and $1$ in the Milstein one.\\
Nevertheless the theorem gives no information on the behaviour of the numerical approximations when we fix the step size $h$ and we vary
the final time $T$. In the standard proof of Theorem \ref{theorem_convergence} one estimates the constants $C(T,\mu,\sigma)$ and
$\bar{C}(T,\mu,\sigma)$ by proving that there exist two positive constants $K(\mu,\sigma), K'(\mu,\sigma)$ such that
$C(T,\mu,\sigma)=\exp(T \cdot K(\mu,\sigma))$ and $\bar{C}(T,\mu,\sigma)=\exp(T \cdot K'(\mu,\sigma))$, by using Gronwall Lemma. In
some situations the exponential growth of the error is a correct prediction (see for example \cite{Milstein1998}).\\
Of course this fact  does not mean that  in
any case the errors $\epsilon_n$ and $\bar{\epsilon}'_n$  exponentially diverge with the time $T$. Indeed if the SDE $(\mu,\sigma)$ admits an equilibrium distribution it could happen that the two errors remain bounded with
respect to the time $T$. Unfortunately this favorable situation does not happen for any values of the step size $h$, but only for values
within a certain region. The phenomenon  just described is known as the  \emph{stability}  problem for a discretization method of a SDE.
 This problem, and the corresponding definition, is usually stated and tested for some specific SDEs (see e.g. \cite{Higham2000,Tocino2005}
  for the geometric Brownian motion, see e.g. \cite{Hernandez1992,Saito1996} the Ornstein-Uhlenbeck process, see e.g.
  \cite{Higham2005,Higham2007} for non-linear equations with a Dirac delta equilibrium distribution, and see e.g. \cite{Yuan2004}
  for more general situation).  We will show some numerical examples of the stability phenomenon for general linear SDEs in Section
  \ref{section_numerical}.

\section{Numerical integration via symmetries}\label{section_symmetries_numerical}

\subsection{Invariant numerical algorithms}\label{subsection_invariant}

When a system of ODEs admits  Lie-point symmetries then  invariant numerical algorithms can be constructed (see e.g. \cite{Levi2006,Levi2011,Dorodnitsyn2011,Winternitz2016}). By completeness we recall the definition of an invariant numerical scheme for a system of ODEs, in the simple case of one-step algorithms. The obvious extension for multi-step numerical schemes is immediate. The discretization of an ODEs system  is a function $F:M \times \mathbb{R} \rightarrow M$ such that if $ x_{n},x_{n-1} \in M$ are the $n, n-1$ steps respectively and $\Delta t_n$ is the step size of our discretization we have that
\begin{eqnarray*}
x_n=F(x_{n-1}, \Delta t_n).
\end{eqnarray*}
If $\Phi:M \rightarrow M$ is a diffeomorphism we say that the discretization defined by the map $F$ is \emph{invariant} with respect to the map $\Phi$ if it happens that
$$\Phi(x_n)=F(\Phi(x_{n-1}),\Delta t_n).$$
If we require that the previous property holds for any $x_n \in \mathbb{R}^n$ and for any $\Delta t_n \in \mathbb{R}_+$ we get
\begin{equation}\label{equation_invariant2}
\Phi^{-1}( F(\Phi(x),\Delta t))=F(x,\Delta t)
\end{equation}
for any $x\in M$ and $\Delta t \in \mathbb{R}$. If $\Phi_a$ is an one-parameter group generated by the vector field $Y=Y^i(x) \partial_{x^i}$, by deriving the relation  $\Phi_{-a}(F(\Phi_a(x),\Delta t))=F(x,\Delta t)$ with respect to $a$, we obtain the relation
\begin{equation}\label{equation_invarinat3}
Y^i(F(x,\Delta t))-Y^k\partial_{x^k}(F)(x,\Delta t)=0
\end{equation}
 which guarantees that the discretization $F$ is invariant with respect to $\Phi_a$, generated by
 $Y$.\\
We can extend the previous definition to the case of a SDE in the following way. Let us  discuss an integration scheme which depends only on the time $\Delta t$ and on the Brownian motion $\Delta W^{\alpha}_n, \alpha=1,\dots,m$ (as for example the Euler method). The same discussion for integration methods depending also on $\Delta \mathbb{W}^{\alpha,\beta}_n$ or other random variables (as the Milstein method) is immediate. In the stochastic case the discretization is a map $F:M \times \mathbb{R}
 \times\mathbb{R}^m \rightarrow M$ and we have
 $$x_n=F(x_{n-1},\Delta t, \Delta W^1,...,\Delta W^m).$$
 Equations \refeqn{equation_invariant2} and
 \refeqn{equation_invarinat3} become
 \begin{eqnarray}
 \Phi^{-1}(F(\Phi(x),\Delta t, \Delta W^{\alpha}))&=&F(x,\Delta t,
 \Delta W^{\alpha}),\label{equation_invariant4}\\
Y^i(F(x,\Delta t, \Delta
W^{\alpha}))-Y^k\partial_{x^k}(F)(x,\Delta t, \Delta
W^{\alpha})&=&0.\label{equation_invariant5}
 \end{eqnarray}

 Since Ito integral strongly depends
 on the fact that the approximation is backward (and not
forward), we stress again that it is not  easy  to prove that a given
discretization $X_n$ converges to the real solution
of the SDE $(\mu,\sigma)$. For this reason we give a theorem which
provides a sufficient (and necessary) condition in order that  Euler and Milstein discretizations of a SDE are invariant with
respect to any strong symmetries $Y_1,...,Y_r$.

\begin{theorem}\label{theorem_invariant}
Let $Y_1,...,Y_r$ be strong symmetries of a SDE $(\mu,\sigma)$. When $Y^i_j=Y_j(x^i)$ are polynomials of first degree in $x^1,...,x^n$, then the Euler discretization (or the Milstein discretization) of the SDE $(\mu,\sigma)$ is invariant with respect to $Y_1,...,Y_r$. If for a given $x_0 \in M$, $\spann\{\sigma_1(x_0),\dots,\sigma_{m}(x_0) \}=\mathbb{R}^n$, also the converse holds.
\end{theorem}
\begin{proof}
We give the proof for the Euler discretization because for the Milstein discretization the proof is very similar.
In the case of Euler discretization we have that
$$F^i(x)=x^i+\mu^i(x) \Delta t+ \sigma^i_{\alpha}(x) \Delta W^{\alpha}.$$
The discretization is invariant if and only if
\begin{eqnarray*}
0=Y_j(F^i)(x)- Y^i_j( F(x))&=&+Y^k_j \partial_{x^k}(F^i)(x)-Y^i_j( F(x))\\
&=&Y^i_j(x)+Y^k_j(x) \partial_{x^k}(\mu^i)(x)\Delta t+Y^k_j(x)\partial_ {x^k}(\sigma^{i}_{\alpha})(x)\Delta W^{\alpha}+\\
&&- Y^i_j(x+\mu \Delta t +\sigma_ {\alpha} \Delta W^{\alpha}).
\end{eqnarray*}
Recalling that $Y_j$ is a symmetry for the SDE $(\mu,\sigma)$ and therefore it has to satisfy the determining equations \eqref{determining_equation1} and \eqref{determining_equation2}, we have that
the Euler discretization  is invariant if and only if
\begin{equation}\label{equation_invariant1}
\begin{array}{c}
Y^i_j(x)+\mu^k(x)\partial_{x^k}(Y^i_j)(x)\Delta t+\frac{1}{2}\sum_{\alpha}\sigma^k_{\alpha}\sigma^h_{\alpha}\partial_{x^kx^h}(Y^i_j)(x) \Delta t+ \sigma^k_ {\alpha}(x)\partial_ {x^k}(Y^i_j)(x) \Delta W^{\alpha} =\\
=Y^i_j(x+\mu \Delta t+ \sigma_{\alpha} \Delta W^{\alpha}).
\end{array}
\end{equation}
Suppose that $Y_j^i=B_j^i+C^i_{j,k}x^k$, then
\begin{eqnarray*}
&Y^i_j(x)+\mu^k(x)\partial_{x^k}(Y^i_j)(x)\Delta t+\frac{1}{2}\sum_{\alpha}\sigma^k_{\alpha}\sigma^h_{\alpha}\partial_{x^kx^h}(Y^i_j)(x) \Delta t+ \sigma^k_ {\alpha}(x)\partial_ {x^k}(Y^i_j)(x) \Delta W^{\alpha}= &\\
&=B^i_j+C^i_{j,k}x^k+C^i_{j,k}\mu^k(x)\Delta t+C^i_{j,k}\sigma^k_{\alpha}(x)\Delta W^{\alpha}&\\
&=B^i_j+C^i_{j,k}(x^k+\mu^k(x) \Delta t +\sigma^k_{\alpha}(x) \Delta W^{\alpha})&\\
&=Y^i_j(x+\mu \Delta t +\sigma_{\alpha} \Delta W^{\alpha}).&
\end{eqnarray*}
Conversely, suppose that the Euler discretization is invariant and
so equality \refeqn{equation_invariant1} holds. Let $x_0$ be as in
the hypotheses of the theorem and choose $\Delta t=0$. Then
$$Y^i_j(x_0+\sigma_{\alpha} \Delta W^{\alpha})=Y^i_j(x_0)+ \sigma^k_{\alpha}(x_0)\partial_{x^k}(Y^i_j)(x_0) \Delta W^{\alpha}.$$
Since $\Delta W^{\alpha}$ are arbitrary and $\spann\{\sigma_1(x_0),...\sigma_{m}(x_0) \}=\mathbb{R}^n$, $Y^i_j$ must be of first degree in $x^1,...,x^n$.
\end{proof}

\begin{remark}
The affinity of the coefficients $Y_j^i$ in Theorem 3.1 strongly depends on the fact that Euler and Milstein numerical approximations depend in an affine way from the noise $\Delta t, \Delta W^{\alpha}, \Delta \mathbb{W}^{\alpha,\beta}$. If we coonsider a non affine numerical approximation we can have non affine symmetries  $Y_1,...,Y_r$  (see the discussion below).

\end{remark}

Theorem \ref{theorem_invariant} can be fruitfully applied in the
following way. If $Y_1,...,Y_r$ are strong symmetries of a SDE we
search a diffeomorphism $\Phi:M \rightarrow M' \subset
\mathbb{R}^n$ (i.e. a coordinate change) such that
$\Phi_*(Y_1),...,\Phi_*(Y_r)$ have coefficients of first degree in
the new coordinates system $x'^1,...,x'^n$. We discretize the
transformed SDE $\Phi(\mu,\sigma)$ using the Euler
discretization, obtaining a discretization $\tilde{F}(x',\Delta t,
\Delta W^{\alpha})$ which is invariant with respect to
$\Phi_*(Y_1),...,\Phi_*(Y_r)$. As a consequence the discretization
$F=\Phi(\tilde{F}(\Phi^{-1}(x),\Delta t, \Delta W^{\alpha})$ is
invariant with respect to $Y_1,...,Y_r$. It is easy to prove that if the map $\Phi$ is
Lipschitz we have that the constructed discretization converges in
$L^1$ to the solution, while if the map $\Phi$ is only locally Lipschitz, the weaker convergence in probability can be established.\\
The existence of the diffeomorphism $\Phi$ allowing the application of
Theorem \ref{theorem_invariant} for general $Y_1,...,Y_r$ is not
guaranteed. Furthermore, even when the map $\Phi$ exists, unfortunately in general it is
not unique. Consider for example the following one-dimensional SDE
\begin{equation}\label{equation_invariant6}
dX_t=\left(a\tanh(X_t)-\frac{b^2}{2} \tanh^3(X_t)\right)dt+b \tanh(X_t) dW_t,
\end{equation}
which has
$$Y=\tanh(x)\partial_x$$
as a strong symmetry.
There are many transformations $\Phi$ which are able to put $Y$ with coefficients of first degree, for example the following two transformations:
\begin{eqnarray*}
\Phi_1(x)&=&\sinh(x)\\
\Phi_2(x)&=&\log{|\sinh(x)|}.
\end{eqnarray*}
Indeed we have that
\begin{eqnarray*}
\Phi_{1,*}(Y)=x'_1\partial_{x'_1},
\Phi_{2,*}(Y)=\partial_{x'_2}.
\end{eqnarray*}
While the map $\Phi_1$ transforms equation \refeqn{equation_invariant6} into a geometrical Brownian motion, the transformation $\Phi_2$ reduces equation \refeqn{equation_invariant6} to a Brownian motion with drift. By applying  Euler method by means of $\Phi_1$ we obtain a poor numerical result ( in fact $\Phi_1$ is  not a Lipschitz function and in this circumstance errors are amplified). By exploiting $\Phi_2$ to make the discretization  we obtain instead an exact simulation. The example shows that this first approach strongly depends on the choice of the diffeomorphism $\Phi$ (which has to be invertible in terms of elementary functions). So it is better to have another procedure able to individuate the best coordinate system for performing the SDE discretization.

\subsection{Adapted coordinates and triangular systems}\label{subsection_adapted}

We introduce a further possible use of Lie's symmetries in
the numerical simulation of a SDE which turns out to be
relevant only in the stochastic framework. Indeed  in the deterministic
setting one can obtain a  completely explicit result.

Suppose that $M=M_1 \times M_2$, with standard cartesian
coordinates $x_1^1,...,x_1^r,x_2^1,...,x^{n-r}_2$, and consider the following triangular
SDE
\begin{eqnarray*}
dX^i_{2,t}&=&\mu^i_2(X_{2,t})dt+
\sigma^i_{2,\alpha}(X_{2,t})dW^{\alpha}_t\\
dX^j_{1,t}&=&\mu^j_1(X^1_{1,t},...,X^{i-1}_{1,t},X_{2,t})dt+
\sigma^j_{1,\alpha}(X^1_{1,t},...,X^{i-1}_{1_t},X_{2,t})dW^{\alpha}_t,
\end{eqnarray*}
where $\mu^i_1,\sigma^i_{1,\alpha}$ do not depend on
$x^i_1,...,x^r_1$. The above SDE is triangular  in
the variables $(x^1_1,...,x^r_1)$. By discretizing a triangular SDE $(\mu,\sigma)$ we reasonable aspect a better behavior than in the general case. Furthermore if  $X^1_{2,t},...,X^{n-r}_{2,t}$ can be exactly simulated with
$\sigma^i_{2,\alpha},\mu^i_2$ growing at most polynomially, we can
conjecture that the error grows polynomially with respect to the maximal integration
time $T$. \\

We recall that the triangular property of stochastic systems is closely related with their symmetries and in
particular to SDEs with a solvable Lie algebra of symmetries.
In order to briefly explain the connection between symmetries and
the triangular form of SDEs, we introduce the following definitions (for more details see \cite{DMU2}).

\begin{definition}
A set of vector fields $Y_1,...Y_r$ on $M$ is called regular on $M$ if, for any $x\in M$, the vectors $Y_1(x),...,Y_r(x)$ are linearly independent.
\end{definition}

\begin{definition}\label{definition_solvable_coordinate}
Let $Y_1,...,Y_r$ be a set of regular vector fields on $M$ which are generators of a solvable Lie algebra $\mathcal{G}$. We say that $Y_1,...,Y_r$ are in \emph{canonical  form} if there are $i_1,...,i_l$ such that $i_1+...+i_l=r$ and
$$(Y_1|...|Y_r)=\left(\begin{array}{c|c|c|c}
I_{i_1} & G^1_1(x) & ... & G^1_l(x) \\
\hline
0 & I_{i_2} & ... & G^2_l(x)\\
\hline
\vdots & \ddots & \ddots & \vdots \\
0 & 0 & ... & I_{i_l}\\
\hline
 0 & 0 & 0 & 0 \end{array} \right), $$
where  $G^h_k:M \rightarrow \matr(i_h,i_k)$ are smooth functions.
\end{definition}

\begin{theorem}
Let a SDE $(\mu,\sigma)$  admit $Y_1,...,Y_r$ as strong
symmetries and let us suppose that $Y_1,...,Y_r$ constitute a solvable Lie
algebra in canonical form. Then the SDE $(\mu,\sigma)$ assumes a triangular form
with respect to $x^{1},....,x^r$.
\end{theorem}
\begin{proof}
The proof is an application of the determing equations and Definition \ref{definition_solvable_coordinate} (see \cite{DMU2}).
\end{proof}

As a notable consequence when we have a SDE $(\mu,\sigma)$ admitting a solvable regular Lie algebra $Y_1,...,Y_r$ we can apply a methodology similar to the one proposed in the previous subsection. Indeed we can start by searching a map $\Phi:M \rightarrow M'$ such that $\Phi(Y_1),...,\Phi(Y_r)$ constitute a solvable Lie algebra in canonical form so implying that $\Phi(\mu,\sigma)$ is a triangular SDE. We can discretize $\Phi(\mu,\sigma)$ according with one of standard methods obtaining a discretization $\tilde{F}$. By composing $\tilde{F}$ with $\Phi$ we obtain a discretization $F(x,\Delta t, \Delta W^{\alpha})=\Phi^{-1}(\tilde{F}(\Phi(x),\Delta t , \Delta W^{\alpha})$ which, when $\Phi$ is Lipschitz, has the property of being a more simple  triangular discretization scheme. Differently from  Theorem \ref{theorem_invariant}, in the present situation we can always construct the diffeomorphism $\Phi$, as the following proposition states.

\begin{proposition}\label{theorem_solvable_coordinate}
Let $\mathcal{G}$ be an $r$-dimensional solvable Lie algebra on $M$ such that $\mathcal{G}$ has constant dimension $r$ as a  distribution of  $TM$. Then, for any  $x_0 \in M$, there exist a set of generators $Y_1,...,Y_r$  of $\mathcal{G}$ and a local diffeomorphism $\Phi:U(x_0) \rightarrow M'$, such that $\Phi_*(Y_1),...,\Phi_*(Y_r)$ are generators in canonical form for $\Phi_*(\mathcal{G})$.
\end{proposition}
\begin{proof}
See \cite{DMU2}.
\end{proof}

We conclude by pointing out that for a general solvable Lie algebra $Y_1,...,Y_r$, the map $\Phi$, whose existence is guaranteed by Proposition \ref{theorem_solvable_coordinate}, does not transform $\Phi_*(Y_1),...,\Phi_*(Y_r)$ into a set of vector fields with coefficients of first degree in $x'^1,...,x'^n$. For this reason and by Theorem \ref{theorem_invariant}, the discretization $F$ constructed by using the diffeomorphism $\Phi$ and the usual Euler discretization algorithm is not invariant with respect to $Y_1,...,Y_r$. \\
Nevertheless if we consider solvable Lie algebras satisfying a special relation, then  $\Phi_*(Y_1),...,\Phi_*(Y_r)$ will have coefficients of first degree in $x'^1,...,x'^r$.

\begin{proposition}
Suppose that the Lie algebra $\mathcal{G}=\spann\{Y_1,...,Y_r\}$ is such that $[[\mathcal{G},\mathcal{G}],[\mathcal{G},\mathcal{G}]]=0$. Then  the coefficients of $\Phi_*(Y_1),...,\Phi_*(Y_r)$ are of first degree in $x'^1,...,x'^r$. Moreover one can choose $\Phi$ such that the coefficients of $\Phi_*(Y_1),\dots,\Phi_*(Y_r)$ are of first degree in all the variables $x'^1,...,x'^n$.
\end{proposition}
\begin{proof}
Let us suppose that $Y_1,...,Y_k$ generates $\mathcal{G}^{(1)}=[\mathcal{G},\mathcal{G}]$. Then $\Phi^*(Y_i)=(\delta^l_i)$ for $i=1,...,k$. Using the fact that $[Y_i,\mathcal{G}^{(1)}] \subset \mathcal{G}^{(1)}$ and the fact that $\Phi_*(Y_1),...,\Phi_*(Y_r)$ are in canonical form, we must have that $\Phi_*(Y_{k+1}),...,\Phi_*(Y_r)$ do not depend on $x'^{k+1},...,x'^{r}$ and  their coefficients must be of first degree in $x'^1,...,x'^r$. \\
The second part of the proposition follows from the well known fact that when the vector fields $Z_1,...,Z_r$ generate an integrable distribution, it is possible to choose a local coordinate system such that the coefficients of $Z_1,...,Z_r$ do not depend on $x'^{r+1},...,x'^n$.
\end{proof}

\section{General  linear SDEs}\label{section_linear}

We first consider the one-dimensional linear SDE
\begin{equation}\label{equation_linear}
dX_t=(a X_t + b) dt+ (c X_t+ d)dW_t,
\end{equation}
where $a,b,c,d \in \mathbb{R}$ and we apply the procedure previously presented
in order to obtain a symmetry adapted discretization scheme.\\
Although  it is possible to prove that equation \refeqn{equation_linear} for $ad-bc \not =0$ does not admit strong symmetries (see \cite{DMU}), we can look at equation \refeqn{equation_linear} as a part of a two dimensional system  admitting Lie symmetries. \\
Let us consider the system
\begin{equation} \label{equation_linear2}
\left(\begin{array}{c}
dX_t\\
dZ_t\end{array}\right)=
\left(\begin{array}{c}
aX_t+b\\
aZ_t \end{array} \right)dt+\left(
\begin{array}{c}
cX_t+d\\
cZ_t \end{array}\right)dW_t,
\end{equation}
on $\mathbb{R} \times \mathbb{R}_+=M$, consisting of the original linear equation and the associated homogeneous one.
It is simple to prove, by solving the determining equations \refeqn{determining_equation1} and \refeqn{determining_equation2}, that the system \refeqn{equation_linear2} admits the following two strong symmetries:
\begin{eqnarray*}
Y_1&=&\left(\begin{array}{c}
z\\
0
\end{array}\right)\\
Y_2&=&\left(\begin{array}{c}
0\\
z
\end{array}\right).
\end{eqnarray*}
The more general adapted coordinate system system  for the
symmetries $Y_1,Y_2$ is given by
$$\Phi(x,z)=\left(\begin{array}{c}
\frac{x}{z}+f(z)\\
\log(z)+l
\end{array} \right),$$
where $l \in \mathbb{R}$ and $f:\mathbb{R}_+ \rightarrow
\mathbb{R}$ is a smooth function. Indeed in the coordinate system
$(x',z')^T=\Phi(x,z)$ we have that
\begin{eqnarray*}
Y'_1&=&\Phi_*(Y_1)=\left( \begin{array}{c}
1\\
0
\end{array}\right),\\
Y'_2&=&\Phi_*(Y_2)=\left( \begin{array}{c}
-x'+e^{z'-l}\partial_z(f)(e^{z'-l})+f(e^{z'-l})\\
1
\end{array}\right).
\end{eqnarray*}
In order to guarantee that  the Euler and Milstein discretization schemes are invariant, by Theorem \ref{theorem_invariant} it is sufficient to choose $f(z)=-\frac{k}{z}$ for some constant $k$.\\
In the new coordinates the original two dimensional SDE becomes
\begin{eqnarray}
dX'_t&=&\left(\left(b-cd+ak-c^2k\right)e^{-Z_t'+l}\right)dt+(d+ck)e^{-Z'_t+l}dW_t\label{equation_linear3}\\
dZ'_t&=&\left(a-\frac{c^2}{2}\right)dt+c dW_t.\label{equation_linear4}
\end{eqnarray}
In the following, for simplicity, we consider the discretization scheme only for $l=0$.
The Euler integration scheme becomes:
\begin{eqnarray*}
\left(\begin{array}{c}
Z'_n\\
X'_n\end{array}\right)&=&
\left(\begin{array}{c}
Z'_{n-1}\\
X'_{n-1}\end{array} \right)+
\left(\begin{array}{c}
\left(a-\frac{c^2}{2}\right)\\
\left(b-cd+ak-c^2k\right)e^{-Z'_{n-1}}
\end{array} \right)\Delta t_n+\\
&&+\left(
\begin{array}{c}
c\\
(d+ck)e^{-Z'_{n-1}}
\end{array}\right)\Delta W_n,
\end{eqnarray*}
and the Milstein scheme:
\begin{eqnarray*}
\left(\begin{array}{c}
Z'_n\\
X'_n\end{array}\right)&=& \left(\begin{array}{c}
Z'_{n-1}\\
X'_{n-1}\end{array} \right)+ \left(\begin{array}{c}
\left(a-\frac{c^2}{2}\right)\\
\left(b-\frac{1}{2}cd+ak-\frac{c^2k}{2}\right)e^{-Z'_{n-1}}
\end{array} \right)\Delta t_n+\\
&&+\left(
\begin{array}{c}
c\\
(d+ck)e^{-Z'_{n-1}}
\end{array}\right)\Delta W_n+ \left(\begin{array}{c}
0\\
-(cd+c^2k)e^{-Z'_{n-1}}
\end{array}\right)\frac{(\Delta W_n)^2}{2}
\end{eqnarray*}
We note that when $k=-\frac{d}{c}$ the
two discretization schemes coincide.\\
Coming back to the original problem, in the Euler case we get:
\begin{equation}\label{equation_discretization1}
 X_n=\exp\left(\left(a-\frac{c^2}{2}\right)\Delta t_n+c\Delta W_n\right)\cdot[X_{n-1}+(b-cd+ak-c^2k)\Delta t_n+(d+ck)\Delta W_n-k]+k
\end{equation}
and in the Milstein case we obtain:
\begin{equation}\label{equation_discretization2}
\begin{array}{rcl}
X_n&=&\exp\left(\left(a-\frac{c^2}{2}\right)\Delta t_n+c\Delta W_n\right)\cdot\left[X_{n-1}+\left(b+ak-\frac{cd+c^2k}{2}\right)\Delta t_n+\right.\\
&&\left.+(d+ck)\Delta W_n-\frac{(cd+c^2k)}{2}(\Delta W_n)^2-k\right]+k.
\end{array}
\end{equation}

\begin{remark}\label{remark_explicit}
There is a deep connection between equations \refeqn{equation_discretization1} and \refeqn{equation_discretization2} and the well-known
integration formula for scalar linear SDEs. Indeed the equation \refeqn{equation_linear} admits as solution
\begin{equation}\label{equation_explicit}
X_t=\Phi_t\left(X_0+\int_0^t{\frac{b-cd}{\Phi_s}ds}+\int_0^t{\frac{d}{\Phi_s}dW_s}\right)
\end{equation}
where
$$\Phi_t=\exp\left(\left(a-\frac{c^2}{2}\right)t+cW_t\right).$$
Equation \refeqn{equation_discretization1} and \refeqn{equation_discretization2} can be viewed as the equations obtained by expanding the
integrals in formula \refeqn{equation_explicit} according with stochastic Taylor's Theorem (see \cite{Kloeden}). This fact should not surprise
since the adapted coordinates obtained in Subsection \ref{subsection_adapted} were  introduced exactly to obtain formula
\refeqn{equation_explicit} from equation \refeqn{equation_linear2}. Since the discretizations schemes \refeqn{equation_discretization1} and
\refeqn{equation_discretization2} are closely linked with the exact solution formula of linear SDEs we call them \emph{exact methods} (or exact
discretizations) for the numerical simulation of  linear SDEs.
\end{remark}

Let us now consider the following two dimensional SDE

\begin{eqnarray*}
\left(\begin{array}{c}
dX_t\\
dY_t
\end{array}\right)&=&\left[\alpha\left(\begin{array}{c}
X_t\\
Y_t \end{array}
\right)+ \beta \left(\begin{array}{c}
-Y_t\\
X_t \end{array}
\right)+
\left(\begin{array}{c}
c_1\\
c_2 \end{array}
\right)  \right]dt+\\
&&
+\left[\sigma\left(\begin{array}{c}
X_t\\
Y_t \end{array}
\right)+ \left(\begin{array}{c}
d_1\\
d_2 \end{array}
\right)\right] dW^1_t+
\left[\sigma'\left(\begin{array}{c}
-Y_t\\
X_t \end{array}
\right)+
\left(\begin{array}{c}
e_1\\
e_2 \end{array}
\right)  \right]dW^2_t
\end{eqnarray*}
The previous equation can be solved explicitly. In particular the homogeneous linear part has solution given by (see, e.g. \cite{EckhardPlaten2010})
\begin{eqnarray*}
\Phi_{t,t'}&=&e^{\left(\mu-\frac{\sigma^2}{2} \right)(t-t')+\sigma (W^1_t-W^1_{t'})} \left( 
\begin{array}{c}
\cos(\beta(t-t')+\sigma'(W^2_t-W^2_{t'})) \\
\sin(\beta(t-t')+\sigma'(W^2_t-W^2_{t'})) 
\end{array}\right.\\
&&\left. \begin{array}{c}
-\sin(\beta(t-t')+\sigma'(W^2_t-W^2_{t'}))\\
\cos(\beta(t-t')+\sigma'(W^2_t-W^2_{t'}))

\end{array}\right),
\end{eqnarray*}
where $\mu=\alpha+\frac{\sigma'^2}{2}$. Thus the solution of the initial equation is 
\begin{eqnarray*}
\left(\begin{array}{c}
X_t\\
Y_t
\end{array}\right)&=&\Phi_{t,0}\cdot \left(\begin{array}{c}
X_0\\
Y_0
\end{array}\right)+ \Phi_{t,0} \cdot \left( \int_0^t{(\Phi_{s,0})^{-1}\cdot\left(\begin{array}{c}
c_1-\sigma d_1+\sigma' e_2\\
c_2-\sigma d_2-\sigma' e_1
\end{array}\right)ds}\right.+\\
&&+\left. \int_0^t{(\Phi_{s,0})^{-1}\cdot\left(\begin{array}{c}
d_1\\
d_2
\end{array}\right)dW^1_t}+\int_0^t{(\Phi_{s,0})^{-1}\cdot\left(\begin{array}{c}
e_1\\
e_2
\end{array}\right)dW^2_t}\right)
\end{eqnarray*}
The Euler discretization of the previous equation becomes:
\begin{eqnarray*}
\left(\begin{array}{c}
X_{t_n}\\
Y_{t_n}
\end{array}\right)&=&\Phi_{t_n,t_{n-1}} \cdot \left(\left(\begin{array}{c}
X_{t_{n-1}}\\
Y_{t_{n-1}}
\end{array}\right)+\left(\begin{array}{c}
c_1-\sigma d_1+\sigma' e_2\\
c_2-\sigma d_2-\sigma' e_1
\end{array}\right) \Delta t_n+ \right. \\
&&+\left.
\left(\begin{array}{c}
d_1\\
d_2
\end{array}\right)\Delta W^1_n
+
\left(\begin{array}{c}
e_1\\
e_2
\end{array}\right)\Delta W^2_n
\right) ,
\end{eqnarray*}
where $\Delta t_n=t_n-t_{n-1}$ and $\Delta W^i_n=W^i_{t_n}-W^i_{t_{n-1}}$.

\section{Theoretical estimation of the numerical forward error for linear
SDEs}\label{section_estimate}

We provide an explicit estimation of the forward error associated with the exact numerical schemes proposed in the previous section for simulating a
general linear SDE. The  explicit solution of a  linear SDE is well known and the use of the resolutive formula for its simulation is extensively used, but in the literature, to the best of our knowledge,  there is no explicit estimation of the forward error.\\

\subsection{Enunciates of the Theorems}

Dividing $[0,T]$ in $N$ parts we obtain  $N+1$ instants $t_0=0,t_n=nh,t_N=T$, with $h=\frac{T}{N}$. We denote by $X^{N,T}_t$ the approximate
solution given by exact Euler method, $\bar{X}^{N,T}_t$ the approximate solution with respect to exact Milstein
method and by $X_t$ the exact solution to the linear SDE.  In the following we will omit $T$ where it is possible.\\

\begin{theorem}\label{theorem_1}
For all $t,T \in \mathbb{R}, t\in [0,T]$, we have
$$\epsilon_N=\left(\mathbb{E}[(X_t-X^{N,T}_t)^2]\right)^{1/2}\leq f(T) g(h) h^{1/2},$$
where $h=\frac{T}{N}$, $g$ is a continuous function and $f$ is a strictly positive continuous function such that for $x \rightarrow + \infty$
\begin{eqnarray*}
f(x)=O(1) & \text{ if } & a<-c^2/2\\
f(x)=O(x) & \text{ if } & a=-c^2/2\\
f(x)=O(e^{C(a,c)x}) & \text{ if } & a>-c^2/2,
\end{eqnarray*}
with $C(a,c) \in \mathbb{R}_+$.
\end{theorem}

\begin{theorem}\label{theorem_2}
For all $t,T \in \mathbb{R}, t\in [0,T]$, we have that
$$\bar{\epsilon}_N=\mathbb{E}[|X_t-\bar{X}^{N,T}_t|]\leq \bar{f}(T) \bar{g}(h) h^{1/2},$$
where $h=\frac{T}{N}$, $\bar{g}$ is a continuous function and $f$ is a strictly positive continuous function such that for $x \rightarrow + \infty$
\begin{eqnarray*}
\bar{f}(x)=O(1) & \text{ if } & a <0\\
\bar{f}(x)=O(e^{C'(a,c)x}) & \text{ if } & a \geq 0,
\end{eqnarray*}
with $C'(a,c) \in \mathbb{R}_+$.
\end{theorem}

Before giving the proof of the two previous theorems we propose some remarks. We recall that a  linear SDE with $ad-bc \not =0$ has an equilibrium
distribution if and only if $a-\frac{c^2}{2} < 0$. Furthermore the equilibrium distribution admits a finite first moment if and only if
$a <0$ and  a finite second moment if and only if $a+\frac{c^2}{2} <0$.  Since we approximate the Ito integral up to the order $h^{1/2}$,
the three cases in Theorem \ref{theorem_1} follow from the fact that for giving an estimate of the error in Euler discretization we need
a second moment control. More precisely we can expect a bounded error with respect to $T$ only when the second moment is finite as
$T \rightarrow + \infty$.\\
Since in the Milstein case  a finite first moment sufficies, in the second theorem we obtain that the error does not grow with $T$ when $a < 0$.
 We can obtain an analogous estimate for the Euler method when $d=0$, i.e. in the case in which the Milstein and Euler discretizations coincide (situation similar to the additive-noise-SDEs setting). The use  of only the first moment finitess  for estimating the error has a price: indeed we obtain an $h^{1/2}$ dependence of the error. We remark that the techniques used in the proof of Theorem \ref{theorem_2} exploit
   some ideas from the recent {\it rough path} integration theory (see e.g. \cite{Friz_Haire2014}), and in particular this circumstance explains
   the $\frac{1}{2}$ order of convergence. This fact induces us to conjecture that our proof probably works also in the general rough path framework
   (for example for fractional Brownian motion by following \cite{Friz2014}). If  in Theorem \ref{theorem_2} we do not require an uniform-in-time
    estimate,  we can apply the methods used  in the proof of Theorem \ref{theorem_1} for obtaining  an error convergence of order  $1$.\\
Essentially the above theorems prove that for $a+\frac{c^2}{2} < 0$ and for $ a <0$ respectively, our symmetry adapted discretization methods
are stable for any value of $h$.  In Section \ref{section_numerical} we give a comparison between the stability of the adapted-coordinates
schemes with respect to the standard Euler and Milstein ones, via  numerical simulations.\\
We conclude by noting that Theorem \ref{theorem_1} and Theorem \ref{theorem_2} cannot be deduced in a trivial way from the standard theorems
about the convergence of Euler and Milstein methods (such as Theorem \ref{theorem_convergence}). Indeed the Euler and Milstein discretizations
of equations \refeqn{equation_linear3} and \refeqn{equation_linear4} do not have Lipschitz coefficients. Furthermore even if a given
discretization $(X'_n,Z'_n)$ of the system composed by \refeqn{equation_linear3} and \refeqn{equation_linear4} should converge to the exact
solution in $L^2(\Omega)$, being the coordinate change $\Phi$ ( introduced in Section \ref{section_linear})  not globally Lipschitz, it
 does not imply that the transformed discretization $(X_n,Z_n)$ converges to the exact solution $(X,Z)$ of the equation \refeqn{equation_linear2}
 in $L^2(\Omega)$. Finally, as pointed out in Subsection \ref{subsection_integration}, Theorem \ref{theorem_convergence} does not guarantee
 an uniform-in-time convergence as Theorem \ref{theorem_1} and Theorem \ref{theorem_2} instead state.

For  proving the theorems we need the following two lemmas. The second  allows to avoid very long calculations (see Appendix A).

\begin{lemma}\label{lemma_exponential}
Let $W_t$ be a Brownian motion, $\alpha,\beta \in \mathbb{R}$ and $n \in \mathbb{N}$ then for any $t \in \mathbb{R}_+$
$$\mathbb{E}[\exp(\alpha t+\beta W_t)W_t^n],$$
is a continuous function of $t$ and in particular it is locally bounded.
Moreover we have that
$$\mathbb{E}[\exp(\alpha t+\beta W_t)]=\exp{\left(\alpha+\frac{\beta^2}{2}\right)t}.$$
\end{lemma}
\begin{proof}
The proof is based on the fact that $W_t$ is a normal random variable with zero mean and variance equal to $t$.
\end{proof}

\begin{lemma}\label{lemma_estimate}
Let $F:\mathbb{R}^2 \rightarrow \mathbb{R}$ be a smooth function
such that $F(0,0)=0$ and such that
$$\mathbb{E}\left[|\partial_t(F)(h,W_h)|^{\alpha} \right],\mathbb{E}[\partial_{w}(F)(h,W_h)],\mathbb{E}[|\partial_{ww}(F)(h,W_h)|^{\alpha}]< L(h),$$
for some $\alpha \in 2\mathbb{N}$, for any $h$ and for some
continuous function $L:\mathbb{R} \rightarrow \mathbb{R}_+$. Then
there exists an increasing function $C: \mathbb{R} \rightarrow
\mathbb{R}$ such that
$$\mathbb{E}[|F(h,W_h)|^{\alpha}]\leq C(h) h^{\alpha/2} .$$
If furthermore $\partial_{w}(F)(0,0)=0$ and
$$\mathbb{E}\left[|\partial_{www}(F)(h,W_h)|^{\alpha}
\right],\mathbb{E}[|\partial_{tw}(F)(h,W_h)|^{\alpha}]\leq L(h)$$
there exists an increasing function $C': \mathbb{R} \rightarrow \mathbb{R}$ such that
$$\mathbb{E}[|F(h,W_h)|^{\alpha}]\leq C'(h) h^{\alpha}.$$
\end{lemma}
\begin{proof}
The statements of the lemma derive as special cases from Lemma 5.6.4 and Lemma 5.6.5 in \cite{Kloeden}.
\end{proof}\\

\subsection{Proof of Theorem \ref{theorem_1}}

We consider the case $t=T$. In fact we will find that our estimate is uniform for $t \leq T$. Using the notations in Remark \ref{remark_explicit} we can write $X_T=I_1+I_2$ where
\begin{eqnarray*}
I_1&=&\Phi_T\int_0^T{(b-cd)\Phi^{-1}_sds}\\
I_2&=&\Phi_T \int_0^T{(d)\Phi^{-1}_sdW_s}.
\end{eqnarray*}
Also the approximation $X^N_T$ can be written as the sums of two \emph{integrals} of the form $X^N_T=I^N_1+I^N_2$ where
\begin{eqnarray*}
I^N_1&=&(b-cd)\sum_{i=1}^N\Phi_T\Phi^{-1}_{t_{i-1}}\Delta t_i, \quad
I^N_2=d\sum_{i=1}^N\Phi_T\Phi^{-1}_{t_{i-1}}\Delta W_i.
\end{eqnarray*}
Obviously the strong error $\epsilon_N$ can be estimated by $\|I_1-I^N_1\|_2+\|I_2-I^N_2\|_2$, where hereafter
$\|\cdot\|_{\alpha}=(\mathbb{E}[|\cdot|^{\alpha}])^{1/\alpha}$.

\subsubsection{Estimate of $\|I_1-I^N_1\|_2$}

Setting $\Psi_{s,t}=\Phi_t(\Phi_s)^{-1}$ for any $s < t$, we obtain (with $\Delta t_i=h$)
$$ \|I_1-I^N_1\|_2 =\mathbb{E}\left[\left|\int_0^T(b-cd)\Psi_{t,T}dt-\sum_{i=1}^N(b-cd)\Psi_{t_{i-1},T}h\right|^2\right]^{1/2}$$
$$=\mathbb{E}\left[\left|\sum_{i=1}^N\int_{t_{i-1}}^{t_i}(b-cd) (\Psi_{t,T}- \Psi_{t_{i-1},T})dt\right|^2\right]^{1/2}$$
$$\leq |b-cd|\left(\sum_{i=1}^N\mathbb{E}\left[\left(\int_{t_{i-1}}^{t_i}|\Psi_{t,T}- \Psi_{t_{i-1},T}|dt\right)  ^2\right]^{1/2}\right).$$
By Jensen's inequality
$$\sum_{i=1}^N\mathbb{E}\left[\left(\int_{t_{i-1}}^{t_i}|\Psi_{t,T}- \Psi_{t_{i-1},T}|dt\right)^2\right]^{1/2} \leq
h^{1/2}\sum_{i=1}^N
\left(E\left[\int_{t_{i-1}}^{t_i}(\Psi_{t,T}-\Psi_{t_{i-1},T})^2dt \right]\right)^{1/2}.$$
and by  Fubini theorem we have to calculate
$\mathbb{E}[(\Psi_{t,T}-\Psi_{t_{i-1},T})^2]$. Since
$$\Psi_{s,t}=\exp\left(\left(a-\frac{c^2}{2}\right)(t-s)+c(W_t-W_s)\right).$$
and
$\Psi_{s,t}=\Psi_{s,u}\Psi_{u,t}$ for any $s \leq u \leq t$ we
obtain that
\begin{equation}\label{expectation1}
\mathbb{E}[(\Psi_{t,T}-\Psi_{t_{i-1},T})^2]=\mathbb{E}[(\Psi_{t,T})^2]\mathbb{E}[(1-\Psi_{t_{i-1},t})^2]
\end{equation}
because $\Psi_{t,T}$ and $\Psi_{t_{i-1},t}$ are independent as a consequence of the  Brownian increments
 independence.\\
\noindent
It is simple to note that the function
$$F_1(t-t_i,W_t-W_{t_i})=1-e^{(t-t_i)\left(a-\frac{c^2}{2}\right)+c(W_t-W_{t_i})},$$
satisfies $F_1(0,0)=0$ and, by Lemma \ref{lemma_exponential},
$$\mathbb{E}[\partial_t(F_1)(t-t_i,W_t-W_{t_i})],\mathbb{E}[\partial_{w}(F_1)(t-t_i,W_t-W_{t_i})],\mathbb{E}[\partial_{ww}(F_1)(t-t_i,W_t-W_{t_i})]<+\infty$$
Thus, by Lemma \ref{lemma_estimate}, there exists an increasing
function $C_1(h)$
$$\mathbb{E}\left[(F_1(t-t_i,W_{t}-W_{t_i}))^2\right] \leq C_1(t-t_i)(t-t_i).$$
Using Lemma \ref{lemma_exponential} we get
$$\mathbb{E}\left[\Psi_{t,T}^2\right]=\exp((2a+c^2)(T-t)),$$
obtaining
\begin{equation}\label{first_errorestimate}
\begin{array}{rcl}
\|I_1-I^N_1\|_2 &\leq & |b-cd|\sqrt{C_1(h)} h^{1/2} \sum_{i=1}^N \exp\left(\left(a+\frac{c^2}{2}\right)(T-t_i)\right) h\\ &  \leq & |b-cd| \sqrt{C_1(h)} G_1(T) h^{1/2},
\end{array}
\end{equation}
where
\begin{equation}\label{G_1}
G_1(T)=\int_0^T{\exp\left(\left(a+\frac{c^2}{2}\right)(T-t)\right)dt}=\frac{1}{a+\frac{c^2}{2}}(\exp((a+c^2/2)T)-1).
\end{equation}

\subsubsection{Estimate of $\|I_2-I^N_2\|_2$}\label{subsection_I_2}

We first consider $I_2=(d)\Phi_T\int_0^T{(\Phi_t)^{-1}dW_t}$.
Since Ito integral involves adapted processes we cannot bring
$\Phi_T$ under  the integral sign. However it is possible to take
advantage of the backward integral formulation which allows to
integrate processes that are measurable with respect to the
(future) filtration  $\mathcal{F}^t=\sigma\{W_s|s \in [t,T]\}$. In
particular when $X_s$ is $\mathcal{F}^t$-measurable then
$$\int_0^T{X_s d^+W_s}=\lim_{n \rightarrow + \infty}\left(\sum_{i=1}^n X_{t^n_{i}}(W_{t^n_{i}}-W_{t^n_{i-1}})\right),$$
where $\{t^n_i\}|_i$ is a  sequence of $n$ points partitions of the interval $[0,T]$, having amplitude decreasing to $0$ and the limit is understood in probability.\\
When $F$ is a regular function, $F(W_t,t)$ is a process which is measurable with respect to both the filtrations $\mathcal{F}_t$ and $\mathcal{F}^t$; therefore one can calculate either  $\int_0^T{F(W_t,t)dW_t}$ and $\int_0^T{F(W_t,t)d^+W_t}$.\\
The next well-known lemma says that we can write $I_2$ in terms of a backward
integral, which allows to bring $\Phi_T$ under the integral sign.

\begin{lemma}\label{lemma_backward1}
Let $F:\mathbb{R}^2 \rightarrow \mathbb{R}$ be a $C^2$-function
such that
$$\mathbb{E}[(F(W_t,t))^2]< + \infty.$$
Then
$$\int_0^T{F(W_t,t)dW_t}=\int_0^T{F(W_t,t)d^+W_t}-\int_0^T{\partial_w(F)(W_t,t)dt}.$$
\end{lemma}
\begin{proof}
We report the proof for convenience of the reader (see, e.g., \cite{Pardoux1987}). Setting
$$\tilde{F}(w,t)=\int_0^w{F(u,t)du},$$
since $F$ is $C^2$ then also $\tilde{F}$ is $C^2$. From this fact
one deduces that
\begin{eqnarray*}
\tilde{F}(W_t,t)-\tilde{F}(W_s,s)&=&\int_s^t{F(W_{\tau},\tau)dW_{\tau}}+\int_s^t{\partial_t(\tilde{F})(W_{\tau},\tau)d\tau}\\&&+\frac{1}{2}\int_s^t{\partial_w(F)(W_{\tau},\tau)d\tau}\\
\tilde{F}(W_t,t)-\tilde{F}(W_s,s)&=&\int_s^t{F(W_{\tau},\tau)d^+W_{\tau}}+\int_s^t{\partial_t(\tilde{F})(W_{\tau},\tau)d\tau}\\
&&-\frac{1}{2}\int_s^t{\partial_w(F)(W_{\tau},\tau)d\tau}.
\end{eqnarray*}
By equating the two expressions one obtains the final formula.
\end{proof}
\begin{remark}
Since in the  proof of Lemma \ref{lemma_backward1} the backward Ito integral plays a fundamental role, the same result  cannot be obtained within a pathwise approach such as the rough paths one.
\end{remark}
\noindent Since
$$(\Phi_t)^{-1}=\exp(-(a-c^2/2)t-cW_t)=F(W_t,t),$$
and  $\partial_w(F)(w,t)=-cF(w,t)$, by Lemma
\ref{lemma_backward1}, we can write
\begin{eqnarray*}
I_2&=&\Phi_T(d)\int_0^T{(\Phi_t)^{-1}dW_t}\\
&=&\Phi_T(d)\left(\int_0^T{(\Phi_t)^{-1}d^+W_t}+c\int_0^T{(\Phi_t)^{-1}dt}\right)\\
&=&(d)\left(\int_0^T{\Psi_{t,T}d^+W_t}+c\int_0^T{\Psi_{t,T}dt}\right).
\end{eqnarray*}
Introducing
$\tilde{I}_2=(d)\int_0^T{\Psi_{t,T}d^+W_t}$ and
$$\tilde{I}^N_2=(d)\sum_{i=1}^N{\Psi_{t_{i},T}\Delta W_i},$$
we have that
\begin{equation}\label{error_decomposition}
\|I_2-I^N_2\|_2 \leq \|\tilde{I}_2-\tilde{I}^N_2\|_2+\left\|(\tilde{I}^N_2-I^N_2)+cd\int_0^T{\Psi_{t,T}dt}\right\|_2.
\end{equation}
We first consider the term $\|\tilde{I}_2-\tilde{I}^N_2\|_2$. The
process $\tilde{I}^N_2$ can be written as $\int_0^T{(d)H_tdW^{+}_t}$
where $H_t$ is the $\mathcal{F}^t-$  measurable process given by
$$H_t=\sum_{i=1}^N \Psi_{t_{i},T} 1_{(t_{i-1},t_i]}(t),$$
where $1_{(t_{i-1},t_i]}$ is the characteristic function of the
interval $(t_{i-1},t_i]$. By Ito's isometry and  Fubini's
Theorem we obtain
\begin{equation}\label{first_term}
\begin{array}{rcl}
\|\tilde{I}_2-\tilde{I}^N_2\|^2_2&=&(d)^2\mathbb{E}\left[\left(\int_0^T{(\Psi_{t,T}-H_t)dW_t}\right)^2\right]\\
&=&(d)^2\mathbb{E}\left[\int_0^T{(\Psi_{t,T}-H_t)^2dt}\right]\\
&=&(d)^2\int_0^T{\mathbb{E}[(\Psi_{t,T}-H_t)^2]dt}\\
&=&(d)^2\sum_{i=1}^N\int_{t_{i-1}}^{t_i}{\mathbb{E}[(\Psi_{t,T}-\Psi_{t_i,T})^2]dt}.
\end{array}
\end{equation}
 Since Brownian motion has independent increments, we have that
\begin{eqnarray*}\label{expectation2}
\mathbb{E}[(\Psi_{t,T}-\Psi_{t_i,T})^2]&=&\mathbb{E}[(\Psi_{t_i,T})^2]\mathbb{E}\left[(1-\Psi_{t,t_i})^2\right].
\end{eqnarray*}
Introducing the function:
$$ H(t_i-t,W_{t_i}-W_{t})=1-\Psi_{t,t_i}$$
which satisfies $H(0,0)=0$, by Lemma \ref{lemma_estimate} and Lemma \ref{lemma_exponential} we obtain 
\begin{eqnarray*}
\|\tilde{I}_2-\tilde{I}^N_2\|^2_2
&\leq&(d)^2\sum_{i=1}^N\exp((2a+c^2)(T-t_i))C_2(h)h^{2}\\
\end{eqnarray*}
where $C_2(h)$ is an increasing function and, finally,
\begin{equation}\label{second_errorestimate}
\|\tilde{I}_2-\tilde{I}^N_2\|_2
\leq (d)\sqrt{(G_2(T)C_2(h))}h^{1/2}
\end{equation}
where
\begin{equation}\label{G_2}
G_2(T)=\int_0^T\exp{(2a+c^2)(T-t)}dt.
\end{equation}\\
In order to estimate the other term in the right-hand side of \eqref{error_decomposition} we note that by introducing
$$K_i(t,W_t)=\exp\left(\left(a-\frac{c^2}{2}\right)(T-t)+c(W_T-W_t)\right)(W_{t_i}-W_t)$$
we have 
\begin{eqnarray*}
I^N_2&=&d\sum_{i=1}^N K_i(t_{i-1},W_{t_{i-1}}),
\end{eqnarray*}
and
$$
K_i(t_{i},W_{t_i})=0$$

\noindent By applying Lemma \ref{lemma_backward1} to $K_i(t_{i},W_{t_i})$ we can write
\begin{eqnarray*}
0-K_i(t,W_t)&=&\int_{t}^{t_{i}}{\partial_{w}(K_i)(s,W_s)d^+W_s}+\int_{t}^{t_i}{\partial_s(K_i)(s,W_s)ds}+\\
&&-c\int_t^{t_i}{\Psi_{s,T}ds}-\frac{c^2}{2}\int_{t}^{t_i}{K_i(s,W_s)ds}.
\end{eqnarray*}
From the previous equality,  by Ito isometry and Minkowski's integral inequality we get
\begin{eqnarray*}
\left\|\tilde{I}^N_2-I^N_2+cd \int_0^T{\Psi_{t,T}dt}
\right\|_2&=&d\left\|\sum_{i=1}^N
\int_{t_{i-1}}^{t_i}{\Psi_{t_i,T}d^+W_t}+\int_{t_{i-1}}^{t_i}{\partial_{w}(K_i)(t,W_t)d^+W_t}+\right.\\
&&\left.+\int_{t_{i-1}}^{t_i}{\partial_t(K_i)(t,W_t)dt}-\frac{c^2}{2}\int_{t_{i-1}}^{t_i}{K_i(t,W_t)dt}\right\|_2\\
&\leq & d \left(\left\|\int_0^T{R_t d^+W_t} \right\|_2+\left\|\int_0^T{M_tdt}\right\|_2\right),\\
&\leq & d
\left(\left(\int_0^T{\mathbb{E}[R_t^2]dt}\right)^{1/2}+\int_0^T{\left(\mathbb{E}[M_t^2]\right)^{1/2}dt} \right),
\end{eqnarray*}
where
\begin{eqnarray*}
R_t&=&\sum_{i=1}^N{(\partial_{w}(K_i)(t,W_t)+\Psi_{t_i,T})1_{(t_{i_-1},t_i]}}(t)\\
M_t&=&\sum_{i=1}^N{\left(\partial_t(K_i)(t,W_t)-\frac{c^2}{2}K_i(t,W_t)\right)1_{[t_{i_-1},t_i]}}(t)
\end{eqnarray*}
When $t_{i-1} < t \leq t_i$,  by independence
\begin{eqnarray*}
\mathbb{E}[R_t^2]&\leq&2\mathbb{E}[\Psi_{t_i,T}^2]\mathbb{E}[(c\Psi_{t,t_i}(W_{t_i}-W_t))^2+(\Psi_{t,t_i}-1)^2].
\end{eqnarray*}
Introducing
\begin{eqnarray*}
F_2(t_i-t,W_{t_i}-W_t)=c\exp\left(\left(a-\frac{c^2}{2}\right)(t_i-t)+c(W_{t_i}-W_t)\right)(W_{t_i}-W_t)\\
F_3(t_i-t,W_{t_i}-W_t)=\exp\left(\left(a-\frac{c^2}{2}\right)(t_i-t)+c(W_{t_i}-W_t)\right)-1,
\end{eqnarray*}
we have that $F_2(0,0)=F_3(0,0)=0$ and $\mathbb{E}[|\partial_{w}(F_i)(t,W_{t_i}-W_t)|^2]$, $\mathbb{E}[|\partial_{ww}(F_i)(t,W_{t_i}-W_t)|^2]$, $\mathbb{E}[|\partial_{t}(F_i)(t,W_{t_i}-W_t)|^2] \leq L(t_i-t)$ and so, by Lemma \ref{lemma_estimate}, there exist two continuous increasing functions $C_3(t), C_4(t)$ such that
$$\mathbb{E}[R_t^2] \leq 2 \exp\left((2a+{c^2})(T-t_i)\right)(C_3(t_i-t)+C_4(t_i-t))|t_i-t|.$$
Since by independence
$$\mathbb{E}[M_t^2]=\mathbb{E}[(a\Psi_{t,T}(W_{t_i}-W_t))^2]=\mathbb{E}[(\Psi_{t_i,T})^2]\mathbb{E}[(a\Psi_{t,T}(W_{t_i}-W_t))^2]$$
analogously we can prove that there exists an increasing function $C_5$ such that
$$\mathbb{E}[M_t^2]\leq \exp\left(\left(2a+{c^2}\right)(T-t_i)\right) C_5(t-t)|t_i-t|. $$
For the second term in the right-hand side of \eqref{error_decomposition}, we have finally the following estimate
\begin{equation}\label{third_errorestimate}
\left\|\tilde{I}^N_2-I^N_2+cd \int_0^T{\Psi_{t,T}dt}\right\|_2 \leq  d\left\{\sqrt{G_2(T)} (\sqrt{2(C_3(h)+C_4(h))}) +G_1(T)\sqrt{C_5(h)}\right\} h^{1/2},
\end{equation}
where
$G_1(T)$ and $G_2(T)$ are given by \eqref{G_1} and \eqref{G_2} respectively.

\subsection{Proof of Theorem \ref{theorem_2}}

We make the proof only for $a <0$, since in the other case the estimate are equal to the Euler case and can be addressed with the same proof. We introduce the two integrals
\begin{eqnarray*}
\bar{I}^N_1&=&(b-cd)\sum_{i=1}^N\Phi_T\Phi^{-1}_{t_{i-1}}\Delta t_i,\\
\bar{I}^N_2&=&d\sum_{i=1}^N\Phi_T\Phi^{-1}_{t_{i-1}}\Delta W_i-\frac{cd}{2}\sum_{i=1}^{N} \Phi_T \Phi^{-1}_{t_{i-1}} ((\Delta W_i)^2-(\Delta t_i)).
\end{eqnarray*}

\subsubsection{Estimate of $\|I_1-\bar{I}^N_1\|_1$}

First we note that (with $\delta t_i=h$)
\begin{eqnarray*}
\| I_1- \bar{I}^N_1 \|_1&\leq& |b-cd| \sum_{i=1}^N \left\|\Phi_T\int_{t_{i-1}}^{t_i}{\Phi_{t}^{-1}dt}-\Phi_T\Phi_{t_{i-1}}^{-1} h \right\|_1 \\
& \leq & |b-dc| \sum_{i=1}^N \left\| \Psi_{t_i,T} \right\|_{\alpha} \left\| \int_{t_{i-1}}^{t_i}{\Psi_{t,t_i}dt}- \Psi_{t_{i-1},t_i} h \right\|_{2n}\\
&=&|b-dc| \left\| \int_{0}^{h}{(\Psi_{t,h}-\Psi_{0,h})dt}\right\|_{2n} \left(\sum_{i=1}^N \left\| \Psi_{t_i,T} \right\|_{\alpha}\right)
\end{eqnarray*}
where we have taken, $n \in \mathbb{N}$, $\frac{1}{2n}+\frac{1}{\alpha}=1$ and $1 <\alpha <2$ such that $\alpha a + \alpha (\alpha-1)\frac{c^2}{2} \leq 0$ (the last condition guarantees that when $T \rightarrow \infty$ we have $\mathbb{E}[\Psi_{t_i,T}^{\alpha}]\rightarrow 0$).
By Jensen's inequality and Lemma \ref{lemma_estimate} we can derive the following estimate:
\begin{eqnarray*}
\left\| \int_{0}^{h}{(\Psi_{t,h}-\Psi_{0,h})dt}\right\|^{2n}_{2n} &\leq &h^{2n-1} \int_0^h{\mathbb{E}[(\Psi_{t,h}-\Psi_{0,h})^{2n}]}dt\\
&\leq & h^{3n} C_5(h),
\end{eqnarray*}
where $C_5(h)$ is an increasing function and in the last inequality we have used the fact that the function $F_4(t,W_t)=\Psi_{t,h}-\Psi_{0,h}$ is such that $F_4(0,0)=0$.
By Lemma \ref{lemma_exponential}, we have that
$$\|\Psi_{t_i,T}\|_{\alpha}=\exp\left(\left(a+\frac{c^2}{2}(\alpha-1)\right)(T-t_i)\right),$$
and so
\begin{eqnarray*}
\|I_1-\bar{I}^N_1\|_1 &\leq |b-cd|& \sum_{i=1}^N \exp\left(\left(a+\frac{c^2}{2}(\alpha-1)\right)(T-t_i)\right) (C_5(h))^{1/2n} h^{3/2} \\
&\leq & |b-cd| G_4(T) (C_5(h))^{1/2n} h^{1/2}
\end{eqnarray*}
where
\begin{equation}\label{G_4}
G_4(T)=\int_0^T{\exp\left(\left(a+\frac{c^2}{2}(\alpha-1))(T-t)\right)\right)dt}.
\end{equation}

\subsubsection{Estimate of $\|I_2-\bar{I}^N_2\|_1$}

First we note that
\begin{eqnarray*}
\| I_2- \bar{I}^N_2 \|_1&\leq& |d| \sum_{i=1}^N \left\|\Phi_T\int_{t_{i-1}}^{t_i}{\Phi_{t}^{-1}dW_t}-\Phi_T\Phi_{t_{i-1}}^{-1} \Delta W_i+\right.\\
&&\left. +\frac{c}{2} \Phi_T \Phi_{t_{i-1}}^{-1}((\Delta W_i)^2-h) \right\|_1 \\
& \leq & |d| \sum_{i=1}^N \left\| \Psi_{t_i,T} \right\|_{\alpha} \left\| \Phi_{t_i}\int_{t_{i-1}}^{t_i}{\Phi_t^{-1}dW_t}-\Psi_{t_{i-1},t_i} \Delta W_i+\right.\\
&&\left.+ \frac{c}{2} \Psi_{t_{i-1},t_i} ((\Delta W_i)^2-h)  \right\|_{2n}
\end{eqnarray*}
where $\alpha,n$ are as in the previous subsection. We introduce the following notation
\begin{eqnarray*}
I_{2,t_i}&=&\Phi_{t_i}\int_{t_{i-1}}^{t_i}{(\Phi_t)^{-1}dW_t}\\
&=&\Phi_{t_i}\left(\int_{t_{i-1}}^{t_i}{(\Phi_t)^{-1}d^+W_t}+c\int_{t_{i-1}}^{t_i}{(\Phi_t)^{-1}dt}\right)\\
&=&\int_{t_{i-1}}^{t_i}{\Psi_{t,{t_i}}d^+W_t}+c\int_{t_{i-1}}^{t_i}{\Psi_{t,{t_i}}dt},
\end{eqnarray*}
where  we have used Lemma \ref{lemma_backward1} and the fact that $\Psi_{s,t}=\Phi_t(\Phi_s)^{-1}$. By introducing also $\hat{I}_{2,t_i}=\int_{t_{i-1}}^{t_i}{\Psi_{t,t_i}d^+W_t}$ and
\begin{eqnarray*}
\bar{I}^N_{2,t_i}&=&\Psi_{t_{i-1},t_i}\Delta W_i-\frac{c}{2}\Psi_{t_{i-1},t_i} ((\Delta W_i)^2-h)\\
\hat{I}^N_{2,t_i}&=&\Psi_{t_{i},t_i}\Delta W_i+\frac{c}{2}((\Delta W_i)^2-h),
\end{eqnarray*}
we have that
$$\|I_{2,t_i}-\bar{I}^N_{2,t_i}\|_{2n}\leq \|\hat{I}_{2,t_i}-\hat{I}^N_{2,t_i}\|_{2n}+\left\|(\hat{I}^N_{2,t_i}-\bar{I}^N_{2,t_i})+c\int_{t_{i-1}}^{t_i}{\Psi_{t,t_{i}}dt}\right\|_{2n}.$$
It is simple to see that the two norms on the right-hand side of the previous expression do not depend on $t_i$ but only on the difference $h=t_i-t_{i-1}$, so we study the functions (with $\Psi_{t_i,t_i}=1$):
\begin{eqnarray*}
Z_1(h)&=&\|\hat{I}_{2,h}-\hat{I}^N_{2,h}\|^{2n}_{2n}=\left\|\int_0^h{(\Psi_{t,h}-1-c(W_h- W_t))d^+W_t}\right\|^{2n}_{2n}\\
Z_2(h)&=&\left\|(\hat{I}^N_{2,t_i}-\bar{I}^N_{2,t_i})+c\int_{t_{i-1}}^{t_i}{\Psi_{t,t_{i}}dt}\right\|^{2n}_{2n}\\
&=&\left\|(1-\Psi_{0,h})W_h+\frac{c}{2}(\Psi_{0,h}+1)W_h^2-\frac{c}{2}(\Psi_{0,h}+1)h+c\int_0^h{\Psi_{t,h}dt}\right\|^{2n}_{2n}
\end{eqnarray*}
By a well-known consequence of Ito isometry (see, e.g., \cite{Friedman1975}) we can estimate the function $Z_1(h)$ as:
$$Z_1(h)\leq D_n h^{n-1} \int_0^h{\mathbb{E}[(\Psi_{t,h}-1-c(W_h- W_t))^{2n}]dt},$$
where $D_n=(n(2n-1))^n$. Since the function
$$F_5(h-t,W_h-W_t)=\exp\left((a-\frac{c^2}{2})(h-t)+c(W_h-W_t)\right)-1-c(W_h-W_t)$$
satisfies $F_5(0,0)=\partial_w(F_5)(0,0)=0$, by Lemma \ref{lemma_estimate} there exists an increasing function $C_6(h)$ such that
$$Z_1(h) \leq C_6(h) h^{3n}.$$
As far as concerned the function $Z_2(h)$, by introducing
$$K(t,W_t)=(1-\Psi_{t,h})(W_h-W_t)+\frac{c}{2}(\Psi_{t,h}+1)(W_h-W_t)^2-\frac{c}{2}(\Psi_{t,h}+1)(h-t),$$
it is immediate to see that
$$Z_2(h)=\left\|K(0,0)+c\int_0^h{\Psi_{t,h}dt}\right\|^{2n}_{2n}.$$
By applying Lemma \ref{lemma_backward1} to $K(h,W_h)$, and by noting that $K(h,W_h)=0$, we obtain 
$$0-K(0,0)=\int_0^h(\partial_t(K)(t,W_t)-\frac{1}{2}\partial_{ww}(K)(t,W_t)dt+\int_0^h\partial_wK(t,W_t)d^{+}W_t$$
Since we have that $-\partial_t(K)(h,W_h)+\partial_{ww}(K)(h,W_h)/2+c\Psi_{0,h}=0$, and  that $K(h,W_h)=\partial_w(K)(h,W_h)=\partial_{ww}(K)(h,W_h)=0$, by Jensen's inequality, Lemma \ref{lemma_estimate} and by applying the same techniques used for obtaining \eqref{third_errorestimate} we find that
$$Z_2(h)^{1/2n} \leq \left\{ (C_7(h))^{1/2n}+(C_8(h))^{1/2n} \right\} h^{3/2}$$
or, equivalently,
$$Z_2(h) \leq C_9(h)h^{3n},$$
with the obvious definition of the function $C_9(h)$.\\
Finally we have 
\begin{eqnarray*}
\|I^N_2-\bar{I}^N_2\|_1&\leq & |d| (C_6(h)^{1/2n}+C_9(h)^{1/2n})\sum_{i=1}^N \exp\left(\left(a+\frac{c^2}{2}(\alpha-1)\right)(T-t_i)\right)h^{3/2}\\
&\leq &|d| (C_6(h)^{1/2n}+C_9(h)^{1/2n}) G_4(T) h^{1/2},
\end{eqnarray*}
where
$G_4(T)$ is given by \eqref{G_4}.

\section{Numerical examples}\label{section_numerical}

We show some numerical experiments which confirm the theoretical estimate proved in Section \ref{section_estimate} and permit to study other properties of the new discretization methods introduced in  Section \ref{section_linear}. \\
We simulate the linear SDE \refeqn{equation_linear} with coefficients $a=-2$, $b=10$, $c=10$ e $d=10$. The coefficients are such that $a+\frac{c^2}{2} >0$ with $a <0$. This means that the considered linear equation admits an equilibrium probability density with finite first moment  and infinite second moment. The coefficient $d$ has been chosen big enough to put in evidence the noise effect.\\

We make a comparison between the Euler and Milstein methods applied directly to equation \refeqn{equation_linear} and the new exact methods \refeqn{equation_discretization1} and \refeqn{equation_discretization2} with the constants $k=0$ and $k=\frac{-d}{c}=-1$. In particular we observe that when $k=-1$, the schemes \refeqn{equation_discretization1} and \refeqn{equation_discretization2} coincide. We calculate the following two errors:
\begin{itemize}
\item the weak error $E^w=|\mathbb{E}[X_t-X^N_t]|$,
\item the strong error $E^s=\mathbb{E}[|X_t-X^N_t|]$.
\end{itemize}
The weak error is estimated trought the explicit expression
$$\mathbb{E}[X_t]=e^{at},$$
for the first moment of the  linear SDE solution, and by using Monte-Carlo method with $1000000$ paths for calculating $\mathbb{E}[X^N_t]$. The strong error is estimated by exploiting Monte-Carlo simulation of $X_t$ and $X^N_t$ with $1000000$ paths. In order to simulate $X_t$ we apply the Milstein method with a steps-size of $h=0.0001$, for which we have verified that it gives a good approximation of both $\mathbb{E}[X_t]$ and  the equilibrium density for $t \rightarrow + \infty$. Since we use Monte-Carlo methods for estimating $E^w$ and $E^s$, the two errors include both the systematic errors of the considered schemes and the statistical errors of the Monte-Carlo estimate procedure.

\begin{figure}[!ht]
   \begin{minipage}[c]{.1\textwidth}
       \centering
            \includegraphics[height=8cm, width=5.33cm, angle=90]{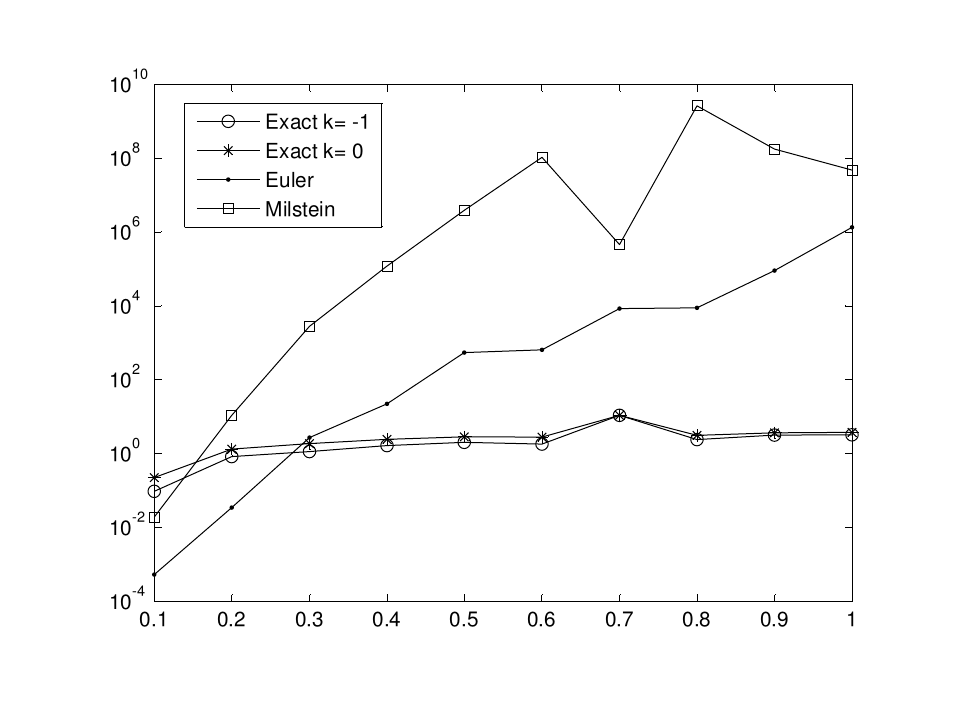}
    \end{minipage}%
    \hspace{50mm}%
    \begin{minipage}[c]{.1\textwidth}
       \centering
        \includegraphics[height=8cm, width=5.33cm, angle=90]{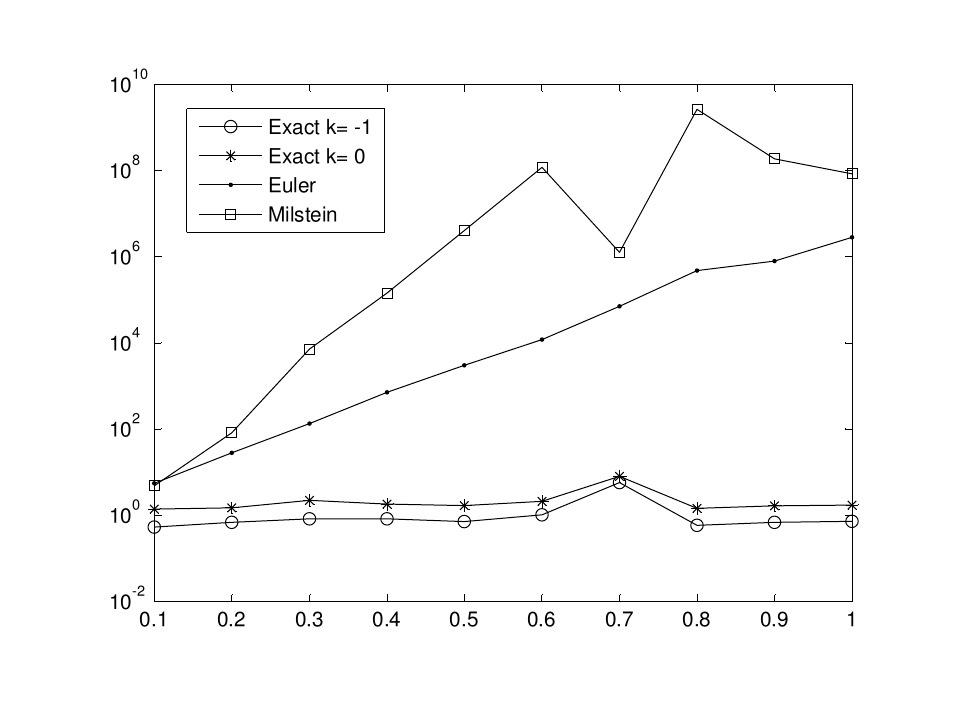}
    \end{minipage}
        \caption{Strong and weak errors with $t \in [0.1,1]$ and stepsize $h=0.025$ \label{figure1}}
\end{figure}

In Figure \ref{figure1} we report the weak and strong errors with respect to the maximum time of integration $t$ which varies from $0.1$ to $1$ and stepsize $h=0.025$. As predicted by Theorem \ref{theorem_2}, the error of the exact method for $k=-1$ remains bounded. It is important to note that for the exact method in the case $k=0$ (where Theorem \ref{theorem_1} and Theorem \ref{theorem_2} do not apply) the errors remains bounded too, while for Euler and Milstein methods the errors grow exponentially with $t$.

\begin{figure}[!ht]
   \begin{minipage}[c]{.1\textwidth}
       \centering
            \includegraphics[height=8cm, width=5.33cm, angle=90]{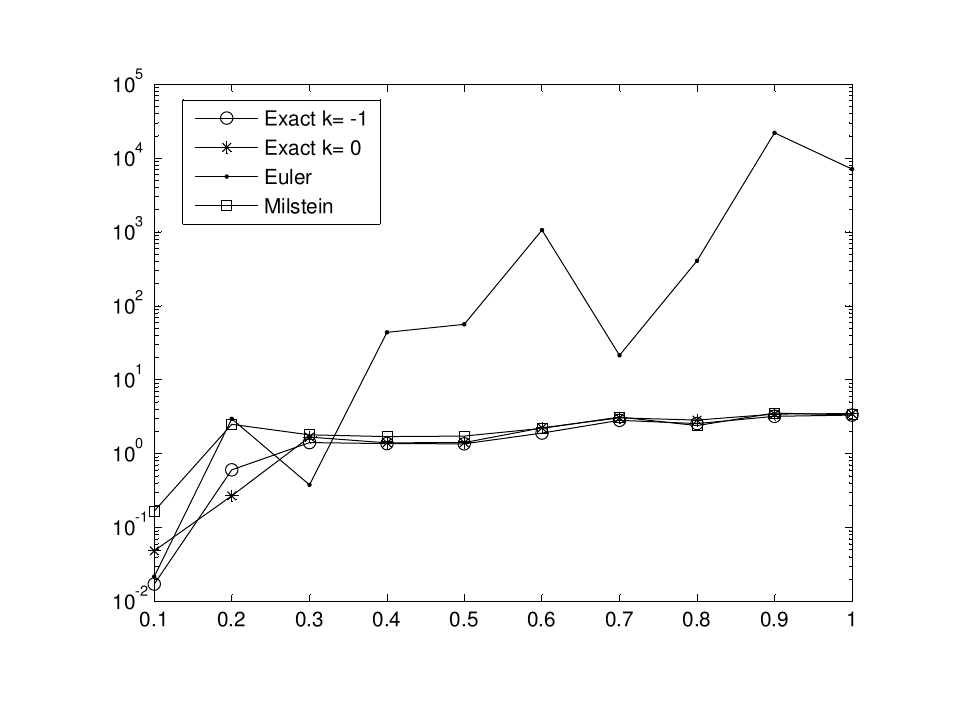}
    \end{minipage}%
    \hspace{50mm}%
    \begin{minipage}[c]{.1\textwidth}
       \centering
        \includegraphics[height=8cm, width=5.33cm, angle=90]{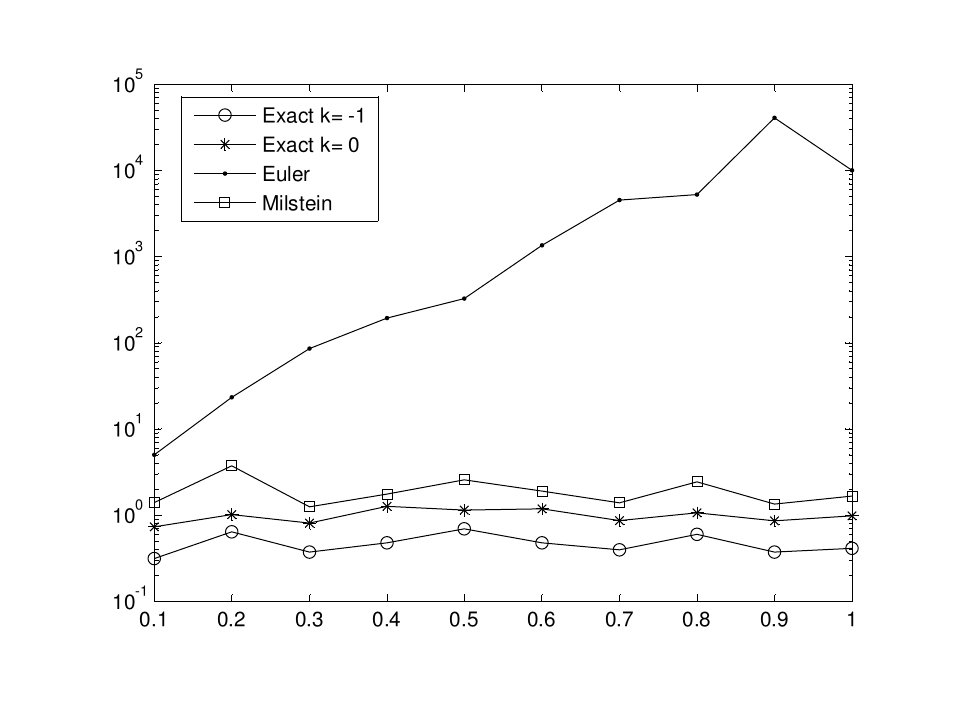}
    \end{minipage}
            \caption{Strong and weak errors with $t \in [0.1,1]$ and stepsize $h=0.01$ \label{figure2}}
\end{figure}

In Figure \ref{figure2} we report the weak and strong errors with respect to the maximum time of integration $t$, which varies from $0.1$ to $1$, and stepsize $h=0.01$. In this situation also the errors of the Mistein method remain bounded. In other words $h =0.01$ belongs to the stability region of the Milstein method but not to the stability region of the Euler method.

\begin{figure}[!ht]
   \begin{minipage}[c]{.1\textwidth}
       \centering
            \includegraphics[height=8cm, width=5.33cm, angle=90]{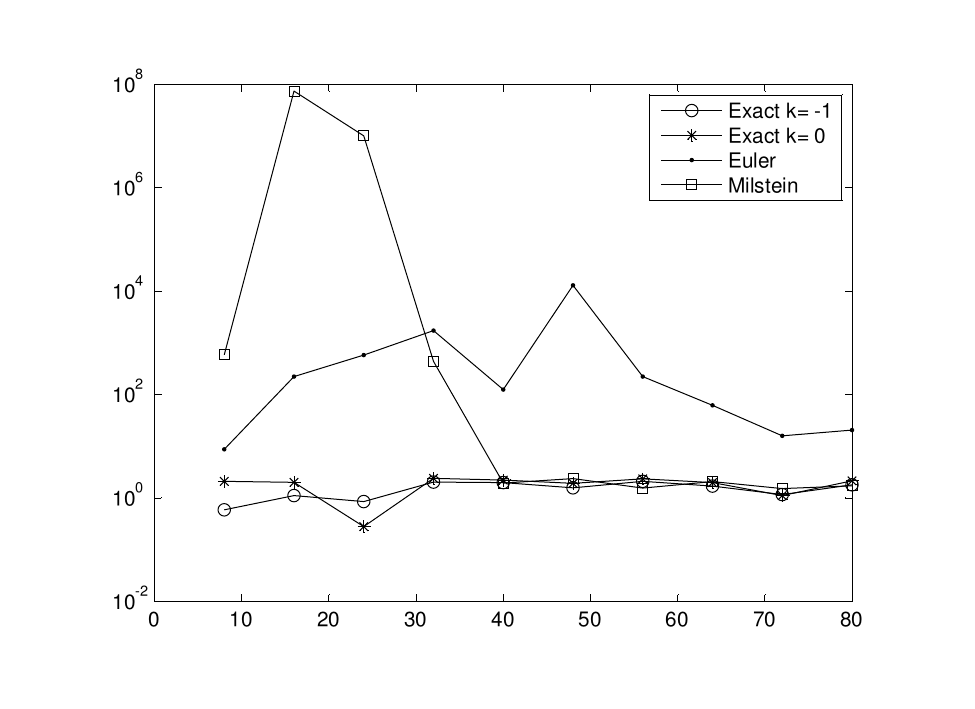}
    \end{minipage}%
    \hspace{50mm}%
    \begin{minipage}[c]{.1\textwidth}
       \centering
        \includegraphics[height=8cm, width=5.33cm, angle=90]{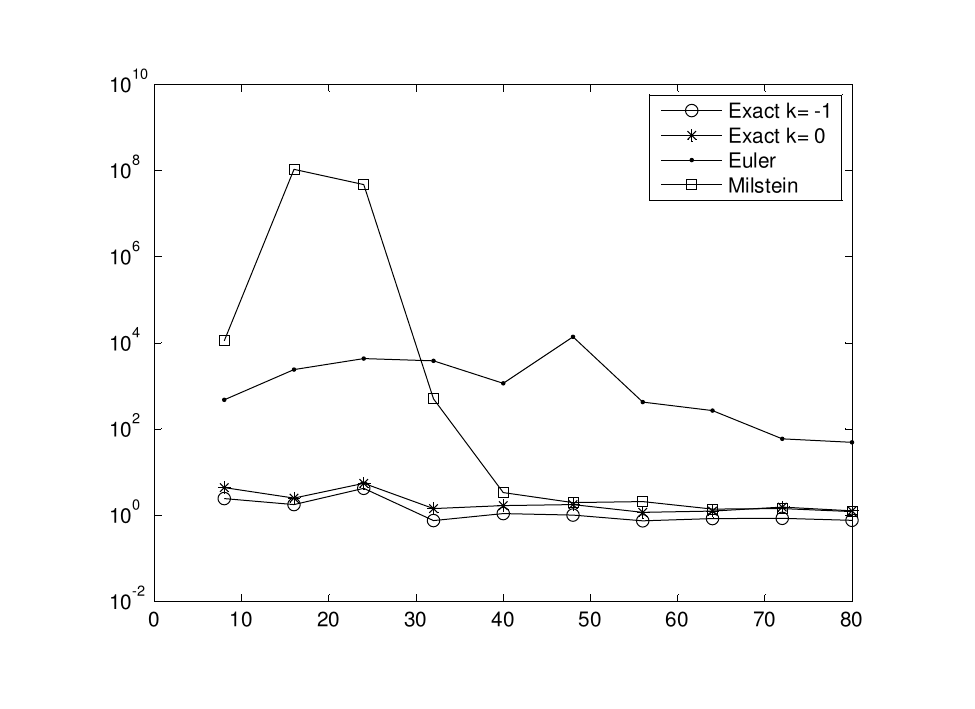}
    \end{minipage}
            \caption{Strong and weak errors with $t=0.5$ and step number $N=[10,80]$ \label{figure3}}
\end{figure}

In Figure \ref{figure3} we plot the weak and strong errors with fixed final time $t=0.5$ and steps number $N=10,...,80$, where the stepsize $h=\frac{t}{N}$. Here we note that the weak and strong errors for the exact methods do not change with the stepsize. This means that with a stepsize of only $h=0.05$ the exact methods have  weak and strong systematic errors less than the statistical errors. Instead for the Milstein scheme the errors grow and only with a stepsize equal to $h=0.0125$ the systematic errors are comparable with the statistical ones. Equivalently we can say that the stability region is $[0,0.0125]$. In the Euler case the systematic error is not  comparable with the statistical one.\\

\begin{figure}[!ht]
      \centering
        \includegraphics[height=8cm, width=5.33cm, angle=90]{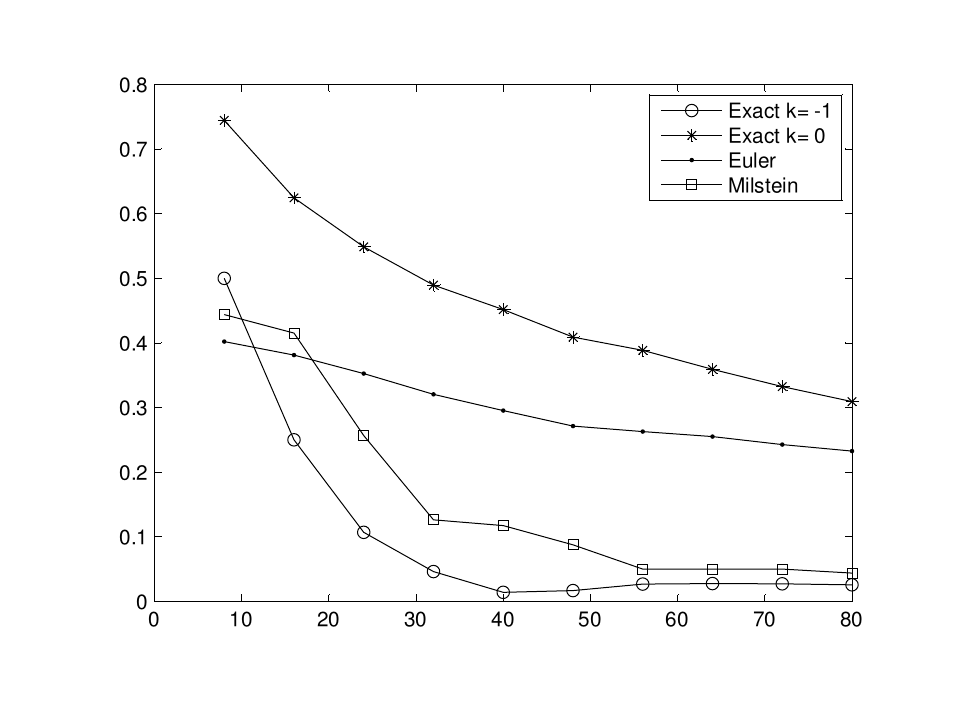}
            \caption{Total variation distance with $t=0.5$ and $h \in [10,80]$ \label{figure4}}
\end{figure}

In Figure \ref{figure4} we report the total variation distance between the empirical probabilities of $X_t$ and  of $X^N_t$ obtained simulating $1000000$  paths. We note that there is a big difference between the exact method for $k=0$ and for $k=-1$. The discrepancy is due to  the fact that the exact method with $k=0$ tends to overestimate the points with probability less then $\frac{-d}{c}$ more than the Euler scheme does.\\
Now we simulate the two dimensional linear SDE analized in Section \ref{section_linear} by

 choosing $\alpha=-20\mbox{, }\beta=-0.5\mbox{, }\sigma=\sigma'=5\mbox{, }\textbf{c}=\textbf{e}=\begin{pmatrix}
0.1 & 0.1
\end{pmatrix}^T$ and $\textbf{d}=\begin{pmatrix}

1 & 1
\end{pmatrix}^T$. Our choice of the parameters guarantees the existence of an equilibrium probability density.\newline
We compare approximated solutions obtained with the Euler method and with our exact method using $h=0.01$. To this end we calculated both the strong and weak componentwise error
\begin{align}
E^w_i&=|\mathbb{E}[X^i_t-X^{i,N}_t]|\\
E^w_i&=\mathbb{E}[|X^i_t-X^{i,N}_t|]
\end{align}
where $X^i_t$ is the $i-$th component of the solution. This time our {\it true} solution is calculated using the Euler method with timestep $h=0.0001$. As in the previous example the error are estimated using a Montecarlo simulation, this time with $10000$ paths, both for the approximated and the {\it true} solution. Again we expect $E^w_i$ and $E^s_i$ to include both systematic and statistical errors.\newline

\begin{figure}[!ht]
   \begin{minipage}[c]{.1\textwidth}
       \centering
            \includegraphics[height=7cm, width=4.66cm, angle=270]{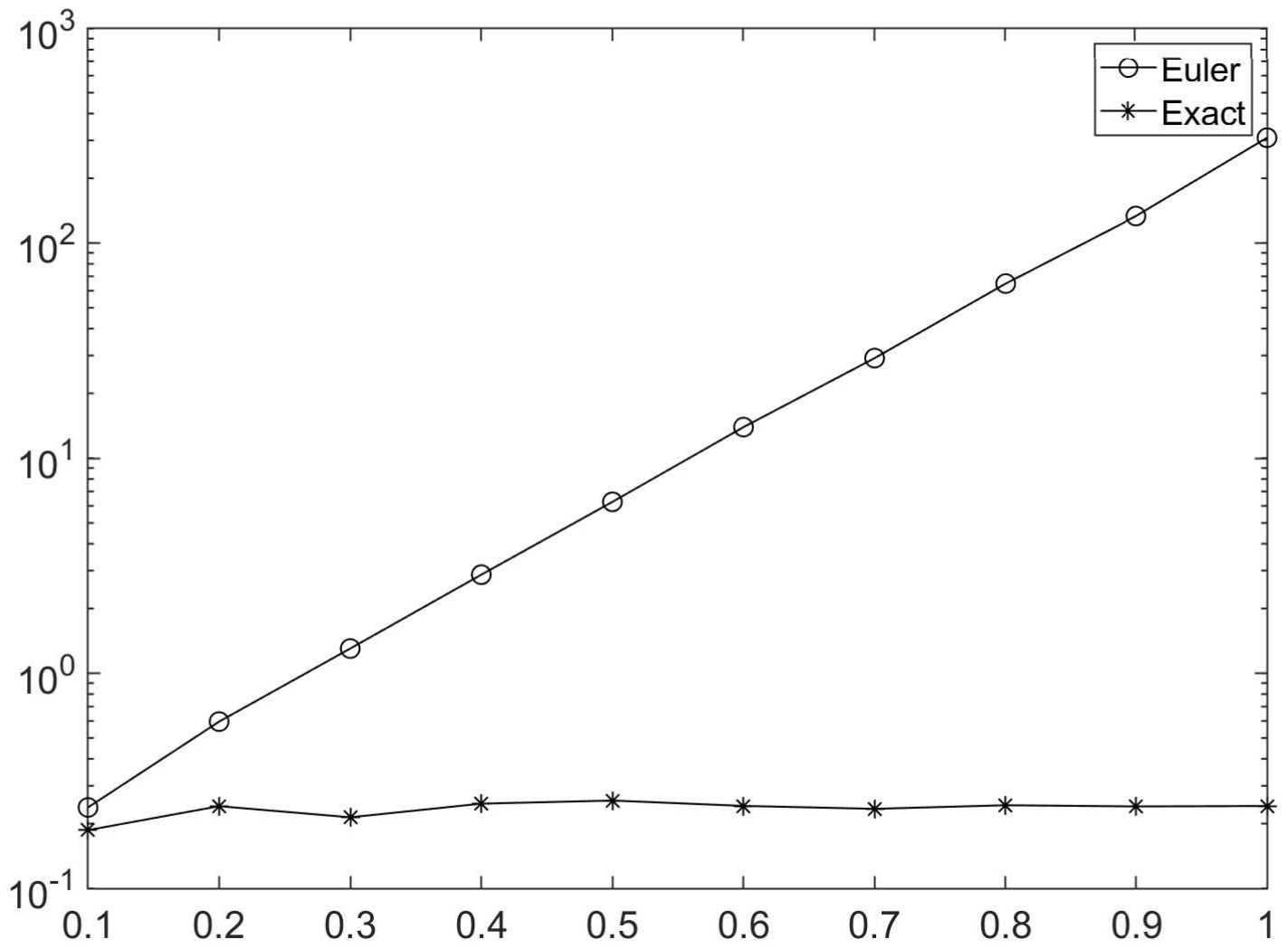}
    \end{minipage}%
    \hspace{50mm}%
    \begin{minipage}[c]{.1\textwidth}
       \centering
        \includegraphics[height=7cm, width=4.66cm, angle=270]{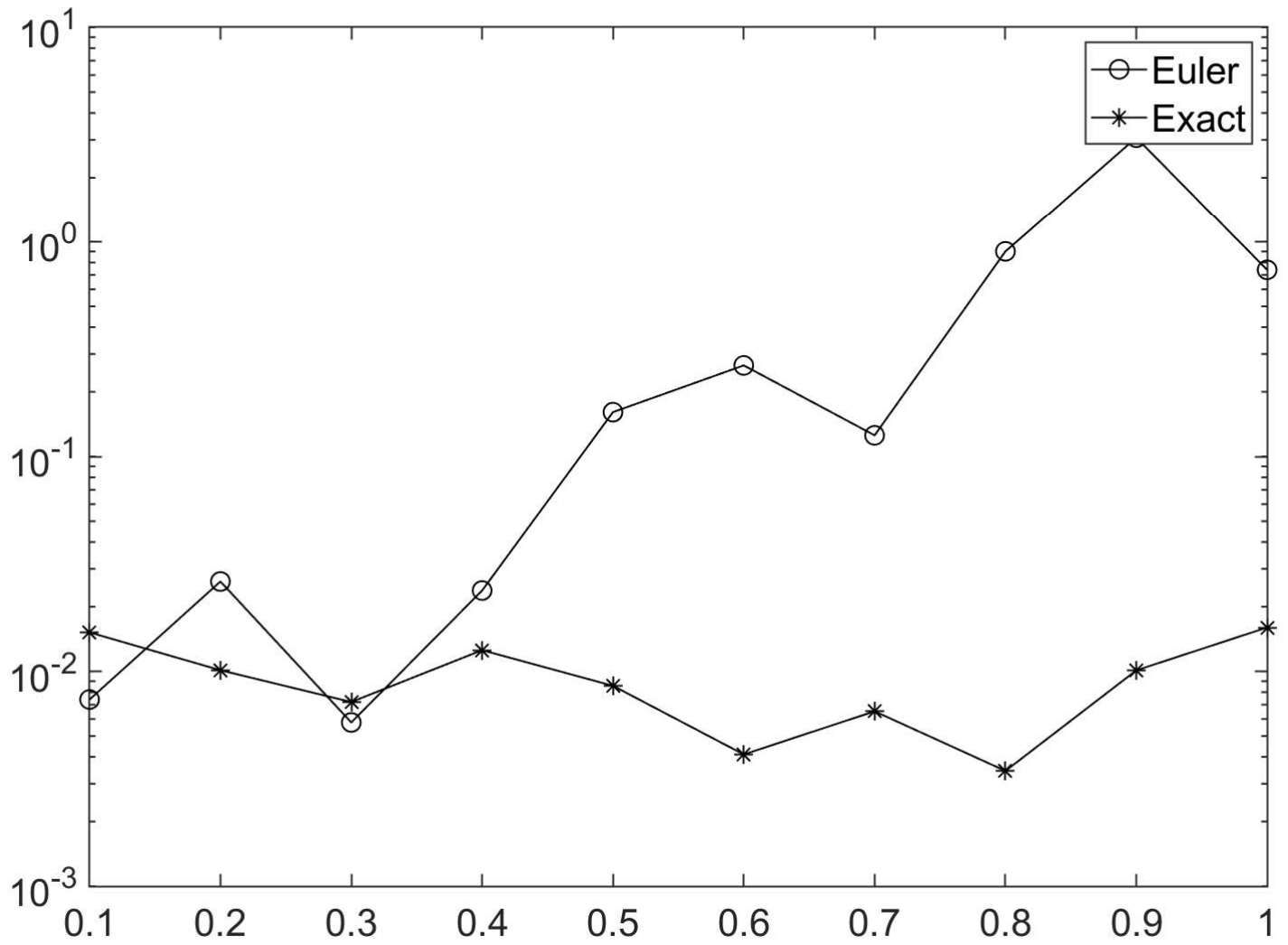}
    \end{minipage}
        \caption{$X_t$ strong and weak errors with $t \in [0.1,1]$ and stepsize $h=0.025$ \label{figure5}}
\end{figure}

\begin{figure}[!ht]
   \begin{minipage}[c]{.1\textwidth}
       \centering
            \includegraphics[height=7cm, width=4.66cm, angle=270]{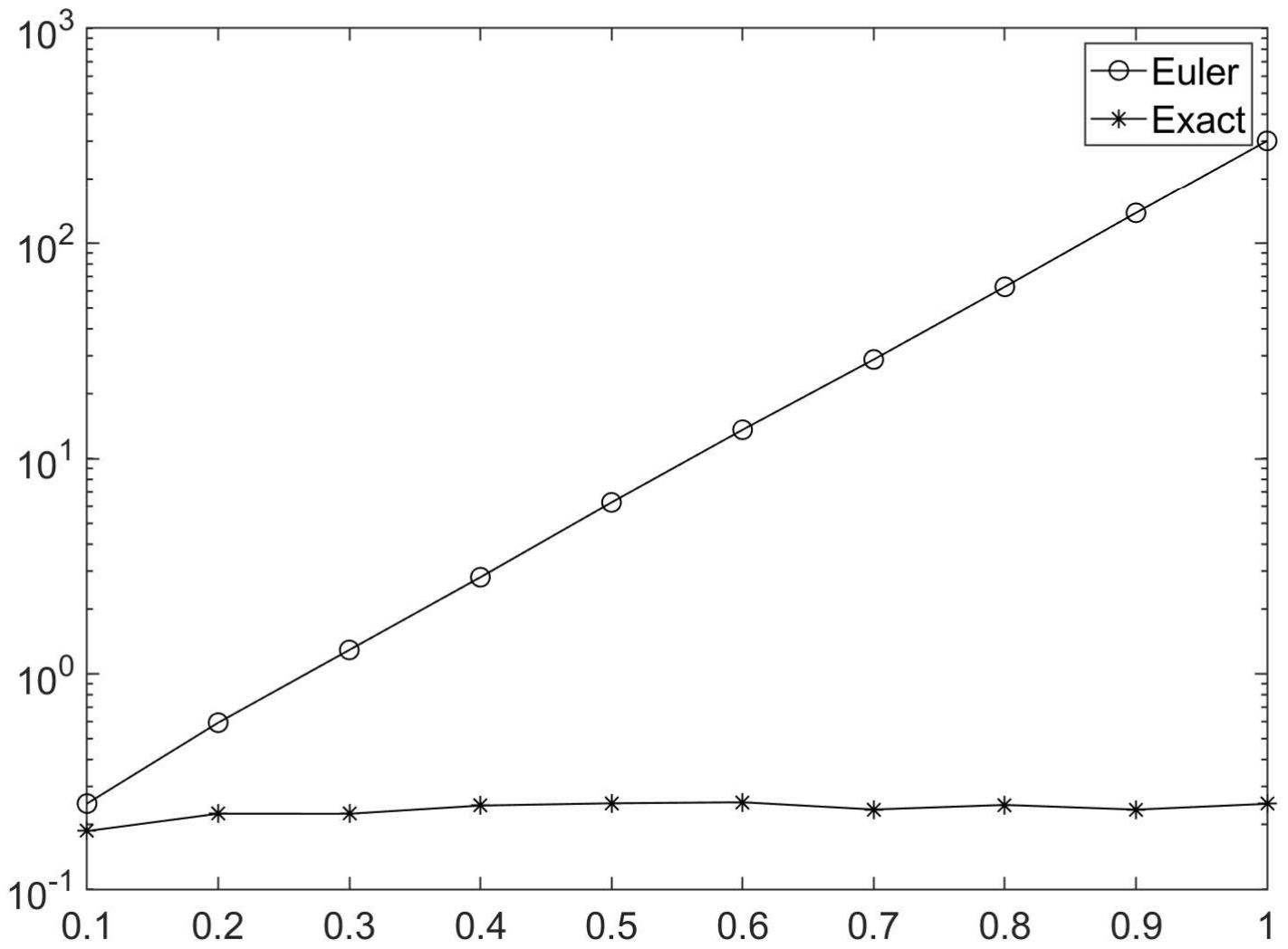}
    \end{minipage}%
    \hspace{50mm}%
    \begin{minipage}[c]{.1\textwidth}
       \centering
        \includegraphics[height=7cm, width=4.66cm, angle=270]{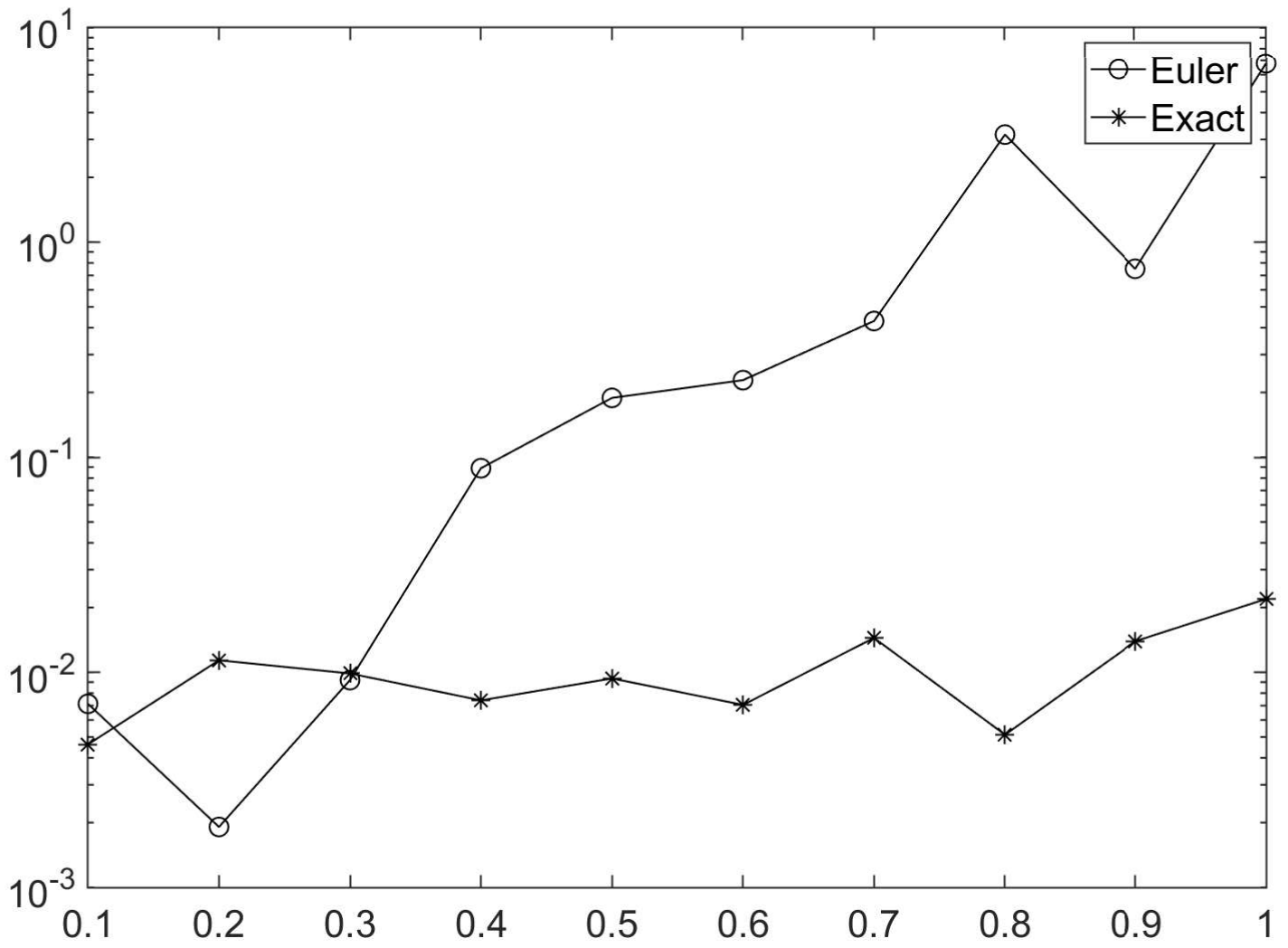}
    \end{minipage}
        \caption{$Y_t$ strong and weak errors with $t \in [0.1,1]$ and stepsize $h=0.025$ \label{figure6}}
\end{figure}

\begin{figure}[!ht]
   \begin{minipage}[c]{.1\textwidth}
       \centering
            \includegraphics[height=7cm, width=4.66cm, angle=270]{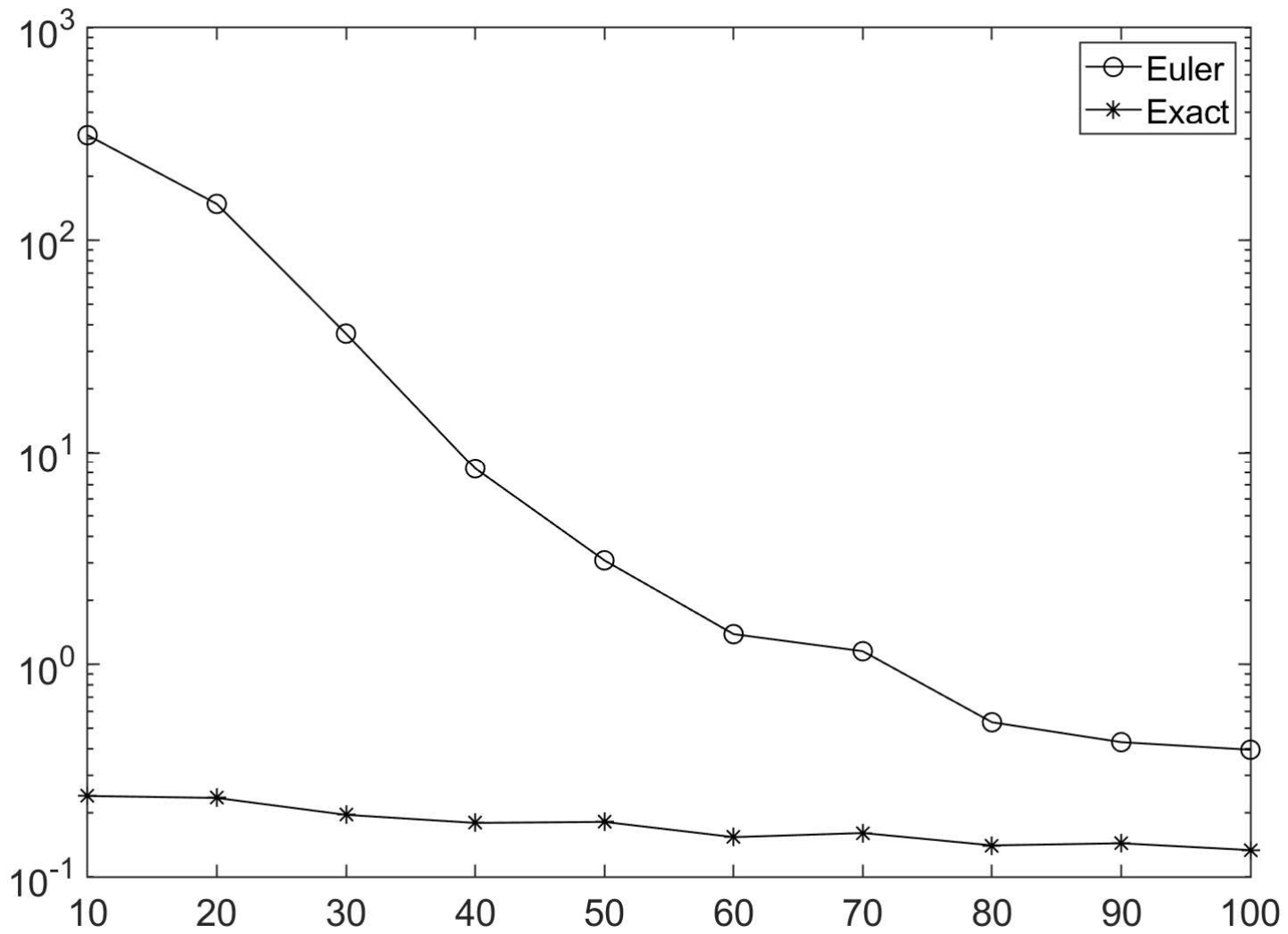}
    \end{minipage}%
    \hspace{50mm}%
    \begin{minipage}[c]{.1\textwidth}
       \centering
        \includegraphics[height=7cm, width=4.66cm, angle=270]{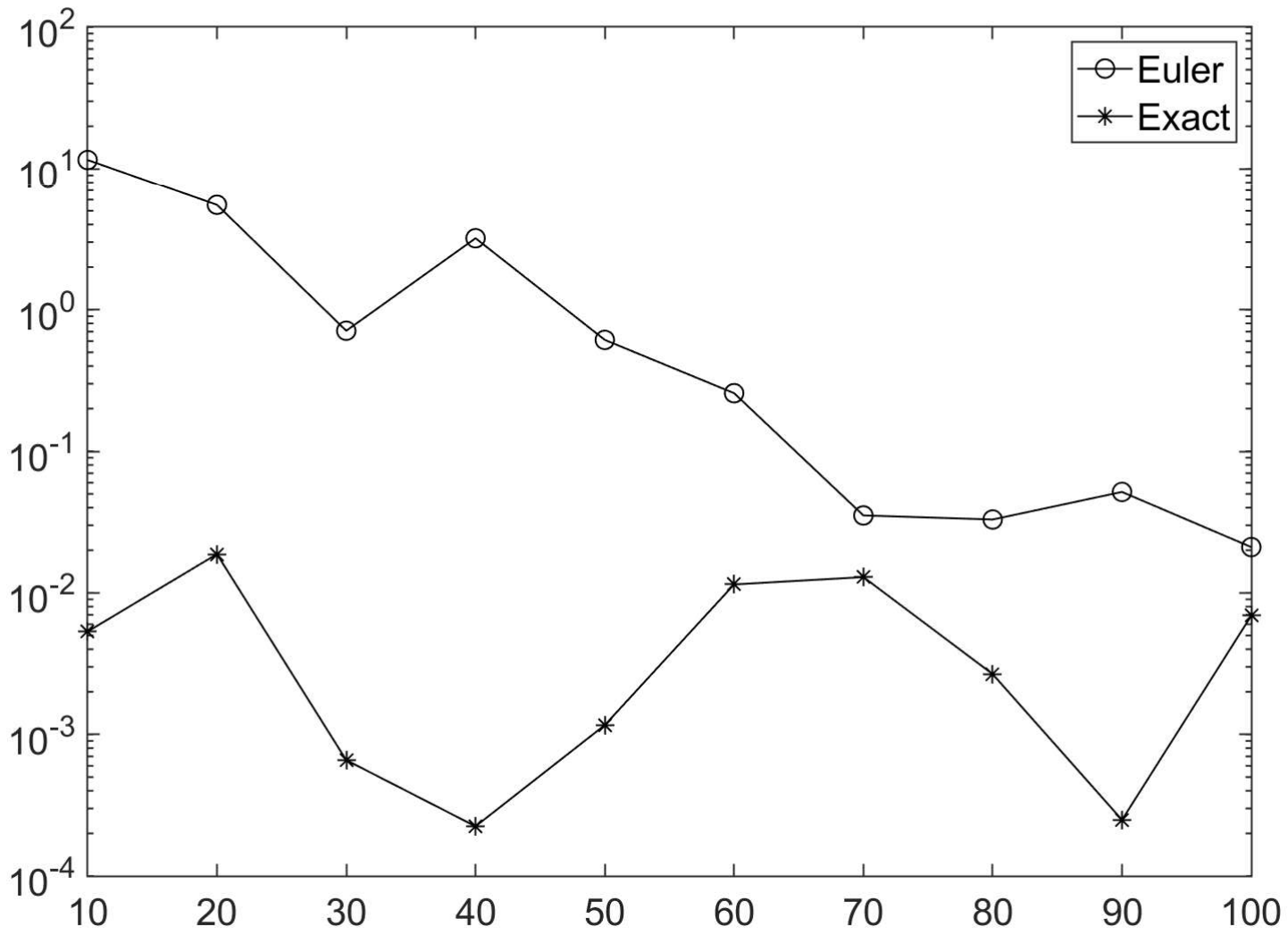}
    \end{minipage}
        \caption{$X_t$ strong and weak errors with $T=1$ and step number $N=[10,100]$ \label{figure7}}
\end{figure}

\begin{figure}[!ht]
   \begin{minipage}[c]{.1\textwidth}
       \centering
            \includegraphics[height=7cm, width=4.66cm, angle=270]{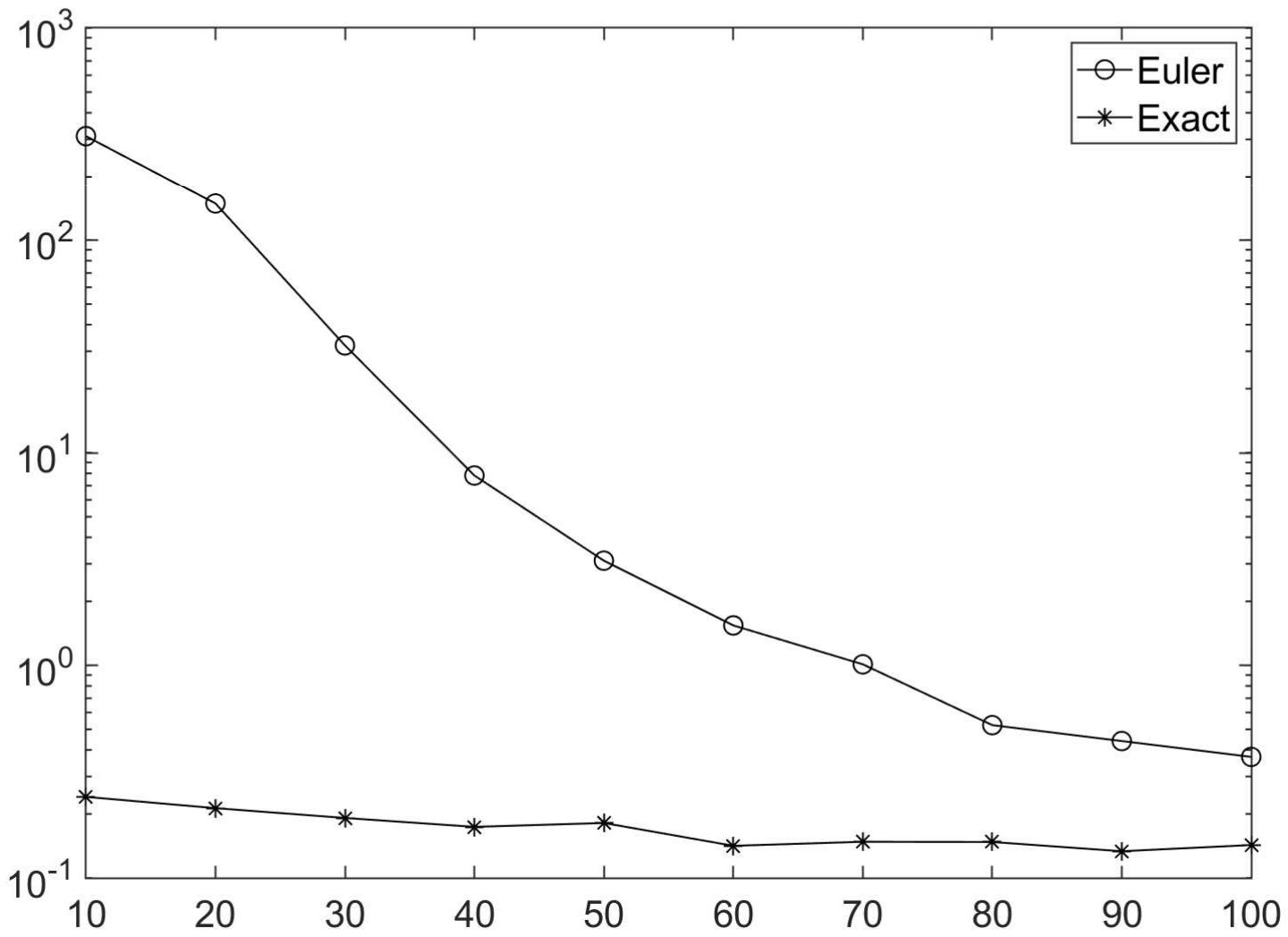}
    \end{minipage}%
    \hspace{50mm}%
    \begin{minipage}[c]{.1\textwidth}
       \centering
        \includegraphics[height=7cm, width=4.66cm, angle=270]{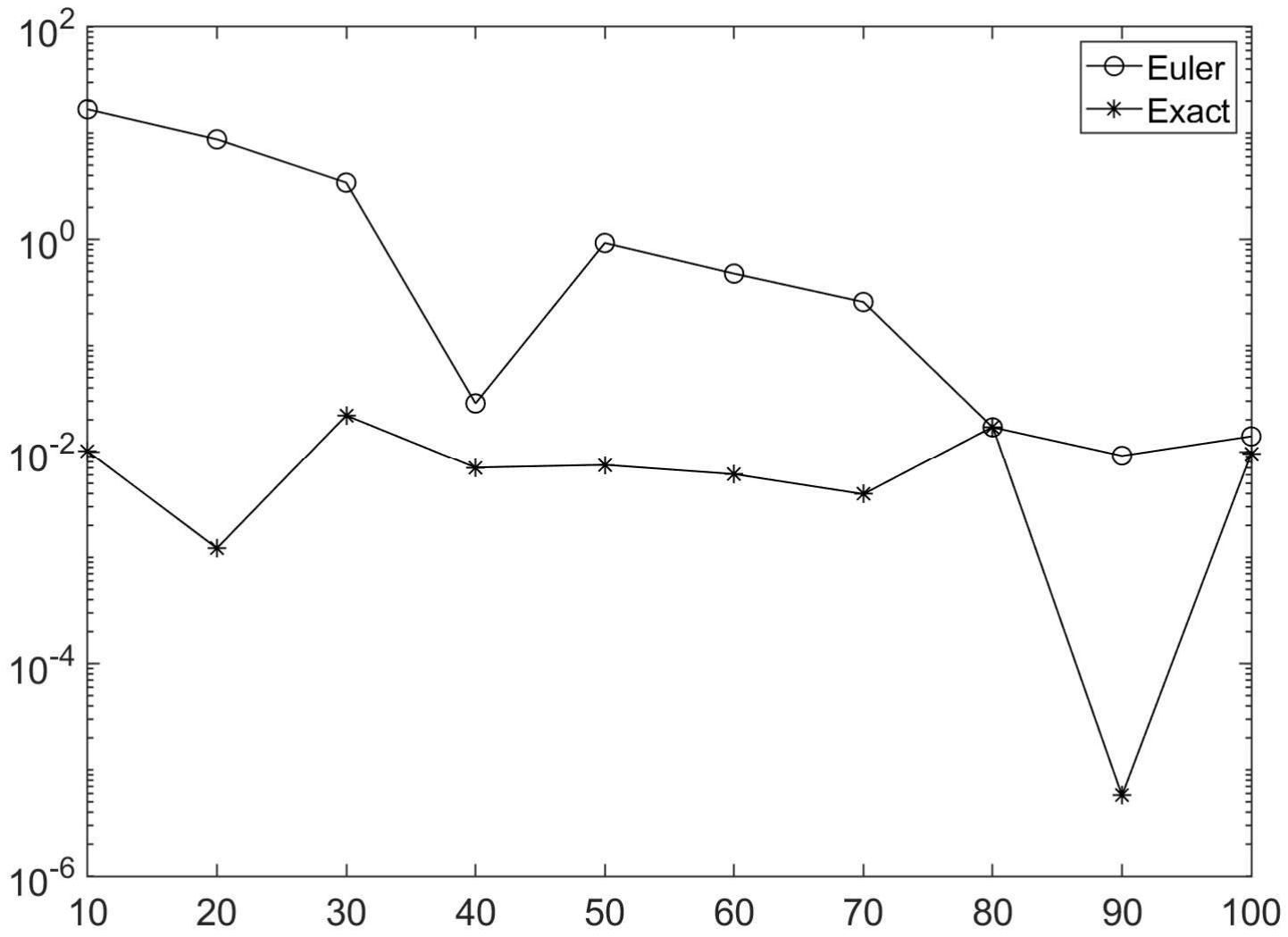}
    \end{minipage}
        \caption{$Y_t$ strong and weak errors with $T=1$ and step number $N=[10,100]$ \label{figure8}}
\end{figure}

In Figure \ref{figure5} and Figure \ref{figure6} we compare the strong and weak error of both components of the simulated solutions with respect to the maximum time of integration varying from $0.1$ to $1$. As can be seen the error from our new method is bounded at all times while the Euler method errors show an exponential growth with respect to the maximum time.\newline

In Figure \ref{figure7} and Figure \ref{figure8} we compare the errors of both approximations for solutions with $T=1$ and timestep size varying between  $0.1$ to $0.01$. As in the previous one-dimensional case we can see how the new exact method gives a good approximation of the {\it true} solution even with large timesteps, while the Euler method fails to achieve the same magnitude of error even using a significative smaller timesteps.

\section{Appendix}

In the proof of Theorem \ref{theorem_1},  by using Lemma \ref{lemma_exponential} and the independence of Brownian increments,  we can estimate the errors in a very explicitely way. In particular without exploiting Lemma \ref{lemma_estimate}. We show main steps and final expressions.\\

\noindent From \eqref{expectation1} we obtain that
$$\int_{t_{i-1}}^{t_i}\mathbb{E}[(\Psi_{t,T})^2]\mathbb{E}[(1-\Psi_{t_{i-1},t})^2]dt=:M_1(h)$$
with
\begin{eqnarray*}
M_1(h)&=&\frac{-a-c^2+h\exp{((2a+c^2)h)}(c^4+3ac^2+2a^2)+(c^2+3a)\exp{((2a+c^2)h)}}{c^4+3ac^2+2a^2}+\\
\\&&+\frac{(2c^2+4a)\exp{(ah)}}{c^4+3ac^2+2a^2}
\end{eqnarray*}
Since $M_1(0)=\partial_hM_1(0)=0$, then $|M_1(h)|\leq M_2(h)h^2$ with $M_2(h):=max_{k\in[0,h]}|\partial_h^2M_1(k)|,$
and, finally,
$$ \|I_1-I^N_1\|_2  \leq |b-cd|h^{1/2}\sqrt{M_2(h)}G_1(T)$$
where $G_1(T)$ is given by \eqref{G_1}, according with \eqref{first_errorestimate}.\\
From \eqref{first_term} we obtain 
\begin{eqnarray*}
\|\tilde{I}_2-\tilde{I}^N_2\|^2_2&=&
(d)^2\sum_{i=1}^N\int_{t_{i-1}}^{t_i} \mathbb{E}[(\Psi_{t_i,T})^2] \mathbb{E}\left[(\Psi_{t,t_i})^2+1-2\Psi_{t,t_i}\right]\\
&=&(d)^2\sum_{i=1}^N\exp{((2a+c^2)(T-t_i))}M_3(h)
\end{eqnarray*}
where
$$ M_3(h)=\frac{3a+2c^2+a\exp{(2a+c^2)}+h(2a+ac^2)-(4a+2c^2)\exp{(ah)}}{2a^2+ac^2}$$
Since $M_3(0)=\partial_hM_3(0)=0$, we have that $|M_3(h)|\leq M_4(h)h^2$ with  $M_4(h):=max_{k\in[0,h]}|\partial_h^2M_3(k)|,$ and
$$ \|\tilde{I}_2-\tilde{I}^N_2\|^2\leq (d)\sqrt{G_2(T)M_4(h)} h^{1/2},$$
according with \eqref{second_errorestimate}.\\
The second term on the right-hand side of \eqref{error_decomposition} becomes
\begin{eqnarray*}
\left\|\tilde{I}^N_2-I^N_2+cd \int_0^T{\Psi_{t,T}dt}
\right\|_2^2&=&d^2\mathbb{E}\left[\left(\sum_{i=1}^N\Psi_{t_i,T}(1-\Psi_{t_{i-1},t_i})(W_{t_i}-W_{t_{i-1}})\right.\right.\\
&&\left.\left.+\sum_{i=1}^N\Psi_{t_i,T}c\int_{t_{i-1}}^{t_i}\Psi_{t,t_i}dt\right)^{2}\right]\\
&=&d^2\left[\sum_{i=1}^N\mathbb{E}[(\Psi_{t_i,T})^2]\mathbb{E}[(K_i+H_i)^2]+\right.\\
&&\left.+2\sum_{i < j}\mathbb{E}[(\Psi_{t_j,T})^2]\mathbb{E}[\Psi_{t_{j-1},t_j}(H_j+K_j)]\mathbb{E}[\Psi_{t_i,t_{j-1}}]\mathbb{E}[(H_i+K_i)]\right]
\end{eqnarray*}
where we have used independence and we have set
$$K_i=(1-\Psi_{t_{i-1},t_i})(W_{t_i}-W_{t_{i-1}}),\quad H_i=c\int_{t_{i-1}}^{t_i}\Psi_{t,t_i}dt$$
We can obtain
\begin{eqnarray*}
M_5(h)&:=&\mathbb{E}[(H_i+K_i)^2]=\exp{(2a+c^2)}(4c^2h^2+h)-2\exp{(ah)}(c^2h^2 +h)+h\\
&&+\frac{c^2(1-\exp{((2a+c^2)h)}}{a(c^2+2a)}+\frac{c^2(\exp{((2a+c^2)h)}-\exp{(ah)}}{a(a+c^2)}\\
&&+2 \left[ -\frac{c^2[(ah-1)\exp{(ah)}+1]}{a^2}+\frac{2c^2[\exp{((2a+c^2)h)}(h(a+c^2)-1)+\exp{(ah)}]}{(a+c^2)^2}\right.\\
&&\left.+\frac{c^2[\exp{((2a+c^2)h)}-\exp{(ah)}(1+h(a+c^2))]}{(a+c^2)^2} \right]
\end{eqnarray*}
and, since $M_5(0)=\partial_h M_5(0)=0$, that $|M_5(h)| \leq M_6(h)h^2$, where $M_6(h):=max_{k\in[0,h]}|\partial_h^2 M_5(k)|.$
Being:
\begin{eqnarray*}
M_7(h)&:=&\mathbb{E}[\Psi_{t_{j-1},t_j}(H_j+K_j)]\\
&=&\frac{c\exp{((2a+c^2)h)}-c\exp{(ah)}+ch(a+c^2)\exp{(ah)}-2ch\exp{((2a+c^2)h)}(a+c^2)}{(a+c^2)}\\
\mathbb{E}[\Psi_{t_i,t_{j-1}}]&=&\exp{(a(t_{j-1}-t_i))}\\
M_8(h)&:=&\mathbb{E}[H_i+K_i]=-ch\exp{(ah)}+\frac{c(\exp{(ah)}-1)}{a},
\end{eqnarray*}
by putting $M_9(h)=M_7(h)M_8(h)$, one can easily verify that $$M_9(0)=\partial_hM_9(0)=\partial^2_hM_9(0)=\partial^3_hM_9(0)=0$$
(because $M_7(0)=\partial_hM_7(0)=M_8(0)=\partial_hM_8(0)=0$) and, therefore,
 $|M_9(h)|\leq M_{10}(h)h^4$, where $M_{10}(h):=max_{k\in[0,h]}|\partial_h^4M_9(k)|.$
Finally
\begin{eqnarray*}
\left\|\tilde{I}^N_2-I^N_2+cd \int_0^T{\Psi_{t,T}dt}
\right\|_2^2  &\leq & d^2\left[\sum_{i=1}^N\exp{((2a+c^2)(T-t_i))}M_6(h)h^2+\right.\\
&&\left.+ 2\sum_{i<j}\exp{((2a+c^2)(T-t_j))}\exp{(a(t_{j-1}-t_i))}M_9(h)\right]\\
&\leq& d^2\left[ G_2(T)M_6(h)h+2M_{10}(h)[ \sum_{i}\exp{((2a+c^2)(T-t_{i+1}))}h^4+\right.\\
&&\left.+\sum_{i< j+1}\exp{((2a+c^2)(T-t_{j}))}\exp{(a(t_{j-1}-t_i))}h^4 \right]
\end{eqnarray*}
that is
$$\left\|\tilde{I}^N_2-I^N_2+cd \int_0^T{\Psi_{t,T}dt}
\right\|_2^2
\leq d^2 \left[ G_2(T)M_6(h)h+2M_{10}(h)(G_2(T)h^3+\bar{G}(T)h^2)\right],
$$
with $$\bar{G}(T)=\int_0^T\int_0^t\exp{((2a+c^2)(T-t)+a(t-s))}dsdt,$$ 

\noindent from which we get:
$$
\left\|\tilde{I}^N_2-I^N_2+cd \int_0^T{\Psi_{t,T}dt}
\right\|_2\leq d\left[\sqrt{G_2(T)M_6(h)+2M_{10}(h)\bar{G}(T)}h^{1/2}+ \sqrt{2M_{10}(h)G_2(T)} h^{3/2}\right],
$$
to be compared with \eqref{third_errorestimate}.

\bibliographystyle{plain}
\bibliography{numerical3}

\end{document}